\documentclass[pdflatex,sn-mathphys-num]{sn-jnl}
\usepackage{natbib}

\usepackage{graphicx}%
\usepackage{multirow}%
\usepackage{amsmath,amssymb,amsfonts}%
\usepackage{amsthm}%
\usepackage{mathrsfs}%
\usepackage[title]{appendix}%
\usepackage{xcolor}%
\usepackage{textcomp}%
\usepackage{manyfoot}%
\usepackage{booktabs}%
\usepackage{algorithm}%
\usepackage{algorithmicx}%
\usepackage{algpseudocode}%
\usepackage{listings}%
\usepackage{xcolor}
\usepackage{fullpage}

\theoremstyle{thmstyleone}%
\newtheorem{theorem}{Theorem}
\newtheorem{proposition}[theorem]{Proposition}

\newtheorem{corollary}[theorem]{Corollary}
\newtheorem{definition}[theorem]{Definition}
\newtheorem{remark}[theorem]{Remark}

\newcommand{\op}[1]{\operatorname{#1}}
\newcommand{\dgm}{\op{dgm}}
\newcommand{\bcd}{\op{bcd}}
\newcommand{\dis}{\op{dis}}
\newcommand{\dB}{\op{d_B}}
\newcommand{\dH}{\op{d_{H}}}
\newcommand{\dGH}{\op{d_{GH}}}

\newcommand{\diag}{\op{diag}}
\newcommand{\SW}{\op{SW}}
\renewcommand{\mp}{\op{mp}}

\newcommand{\score}{\op{score}}
\newcommand{\tr}{\op{tr}}
\newcommand{\norm}[1]{\left\lVert#1\right\rVert}

\newcommand{\vc}{\mathbf{c}}

\newcommand{\vR}{\mathbb{R}}
\newcommand{\vN}{\mathbb{N}}

\newcommand{\minDL}{\texttt{minDL}}
\newcommand{\tol}{\texttt{tol}}

\begin{document}

\title[Article Title]{A Stable Measure of Similarity for Time Series using Persistent Homology}
\author[1]{\fnm{Bala} \sur{Krishnamoorthy}}\email{kbala@wsu.edu}
\author*[1]{\fnm{Elizabeth} \sur{Thompson}}\email{elizabeth.thompson1@wsu.edu}
\affil{
	\orgdiv{Department of Mathematics and Statistics},
	\orgname{Washington State University},
	\orgaddress{
		\street{14204 NE Salmon Creek Ave},
		\city{Vancouver},
		\state{Washington},
		\postcode{98686},
		\country{United States}
	}
}


\abstract{
Persistent homology, the study of holes that appear in data as one thickens balls centered around its points over time, has theoretically guaranteed stability. 
That is, small data perturbations guarantee small changes in the lifetimes of these holes. 
This stability has been used to construct a measure of periodicity for a single univariate time series, denoted the sliding windows and 1-persistence scoring function, $\score(f_1)$. 
Along with periodicity analysis of single time series, similarity between a pair of time series has also been studied. 
One popular measure of similarity between two time series is percent determinism (\%DET), which measures the correlation between two time-series embeddings. 
Percent determinism relies on Euclidean distance alone, which makes it less robust to small time-series perturbations in practice. 
Motivated by the stability of persistent homology, coupled with the persistence-based periodicity scoring function $\score(f_1)$, we introduce a novel persistent-homology based measure of time-series similarity which we denote the bi-conditional periodicity score, $\score(f_1,f_2)$. 
We prove the stability of our measure under small time series and frequency perturbations, as well as the existence of a minimum embedding dimension for the convergence of our score. 
Our latter result implies that larger embedding dimensions may be necessary to reach desired levels of convergence. 
Since pairwise distances between points in these larger dimensions may start to concentrate, we also prove the stability of our measure under dimension reduction which guarantees that as long as the first $K$ principal components capture a majority of the variance under orthogonal projection, the score will undergo small changes. 
We next introduce an algorithm for computing the bi-conditional periodicity score and deduce its computational complexity as $O(N \log N + PK^2 + P^6)$ for $N$ the number of time series points, $P$ the number of embedding points, and $K$ the number of principal components.  
We experimentally verify the greater stability of our measure in comparison with \%DET on both synthetic time series as well as real climate data. 
Along with the greater stability of our measure to \%DET, $\score(f_1,f_2)$ also requires only one parameter for its computation while \%DET requires four.
}

\keywords{time series, persistent homology, PCA}

\maketitle

\section{Introduction}\label{sec:intro}

Time series analysis is a well-studied field. 
One time-series feature commonly studied is the extraction of its dominant frequency components. 
One of the earliest numerical methods introduced for estimating the frequency content of a univariate signal is the Fourier Transform \cite{bracewell2000}, which quantifies the strength of a time series at a given frequency.
Frequency is commonly measured in Hertz, or cycles per second.
Given a continuous, real-valued time series $f(t):T \subseteq \vR \rightarrow \vR$ and a frequency $w$, the continuous Fourier Transform of $f$ at $w$ is the area under the product of $f$ with a purely sinusoidal signal at frequency $w$, given by 

\[ CFT[f](w) = \int_T f(t) \cdot e^{- i 2\pi w t} dt,\] 

\noindent where $e^{-i2\pi wt} = \cos(2\pi w t) - i \sin(2\pi w t)$.
For discrete signals, this transform simply computes the dot product of $f(t)$ with purely sinusoidal signal.
Intuitively, larger values of $CFT[f](w)$ indicate greater presence of this frequency in $f$.
Alternative transforms have also been introduced that are more suited for studying nonstationary series (i.e. those whose frequencies change with time) or series with wavelet structure (i.e. seismic waves), such as the Short Time Fourier Transform \cite{752051} and the Wavelet Transform \cite{russell2016jean}.
A newer method has also been introduced that can analyze the frequency content of a time series by measuring correlation between the sliding windows embeddings $\SW_{M,\tau}f(t)$ and $\SW_{M,\tau}e^{i 2\pi w t}$ of a time series $f$ and a reference signal $e^{i 2\pi w t}$ given an embedding dimension $M \in \vN$ and a time lag $\tau \in \vR$. 
Denoted as cross-recurrence quantification analysis \cite{coco2014cross}, this method summarizes all pairs of points that cross-recur (i.e. are close enough in Euclidean distance) between $\SW_{M,\tau}f(t)$ and $\SW_{M,\tau}e^{i 2\pi w t}$ in a binary matrix $\mathcal{C}$. 
The proportion of cross-recurring states  (i.e. the proportion of ones in $\mathcal{C}$) is denoted as percent recurrence (\%REC), given by

\[ \%\text{REC}[f,e^{- i 2\pi w t}] = \frac{1}{|C|} \sum_{ij} \delta_{\tol}(C[i,j]), \]

\noindent where $\delta_\tol (C[i,j]) = 1$ if $||\SW_{M,\tau}f(t_i) - \SW_{M,\tau}e^{i 2\pi w t_j}|| \leq \tol$ and $|C|$ denotes the total number of points in $C$.
This measure essentially quantifies the correlation between the embeddings of $f$ and $e^{- i 2\pi w t}$. 
Further, the proportion of cross-recurring states that exist consecutively for at least $\minDL$ time steps (i.e. the proportion of ones in $\mathcal{C}$ that belong to a diagonal of at least $\minDL$ ones) denotes the percent determinism (\%DET), given by 

\[ \%\text{DET}[f,e^{- i 2\pi w t}] = \frac{1}{|C|} \sum_{ij} \eta_\minDL \left( \delta_\tol(C[i,j]) \right),  \]

\noindent where $\eta_\minDL \left( \delta_\tol(C[i,j]) \right) = 1$ if $\delta_\tol(C[i,j]) = 1$ and $C[i,j]$ is in a diagonal of at least $\minDL$ ones.

While the Fourier Transform requires less steps for frequency analysis when compared to \%REC, \%REC has the added capability of estimating the correlation between two time series embeddings.
In practice, \%REC is less robust to added noise as it relies on Euclidean distance alone.
We discuss these limitations in more detail in the next section.

\section{Limitations}\label{sec:limitations}

As previously mentioned, the Fourier Transform and cross-recurrence quantification are both tools we can use to study the frequency content of a time series, however in practice \%REC and \%DET are less robust to added data noise.
Unlike the Fourier Transform, both recurrence measures have the additional capability of quantifying time-series similarity.
As we'd naturally expect real time series data to contain some inherent noise, a measure that is more robust is necessary. 
In Figure \ref{fig:instability_FT_DET}, we show an example of the instability of \%REC and \%DET on a synthetic sinusoidal with added Gaussian noise. 
When just increasing the noise in $f$ by 5\%, determinism drops by nearly 20\% and recurrence drops by nearly 2\%.

Persistent homology, the study of holes in different dimensions that appear in point cloud data upon the thickening of their points over time, has become a popularized method for time series analysis and prediction \cite{ichinomiya2025machine,10.1063/5.0102421}.
Part of its growing popularity is due to its guaranteed stability to small amounts of added Gaussian noise to time series, which we discuss in more detail below.

\begin{figure}[ht!]
\centering
\includegraphics[width=0.95\textwidth]{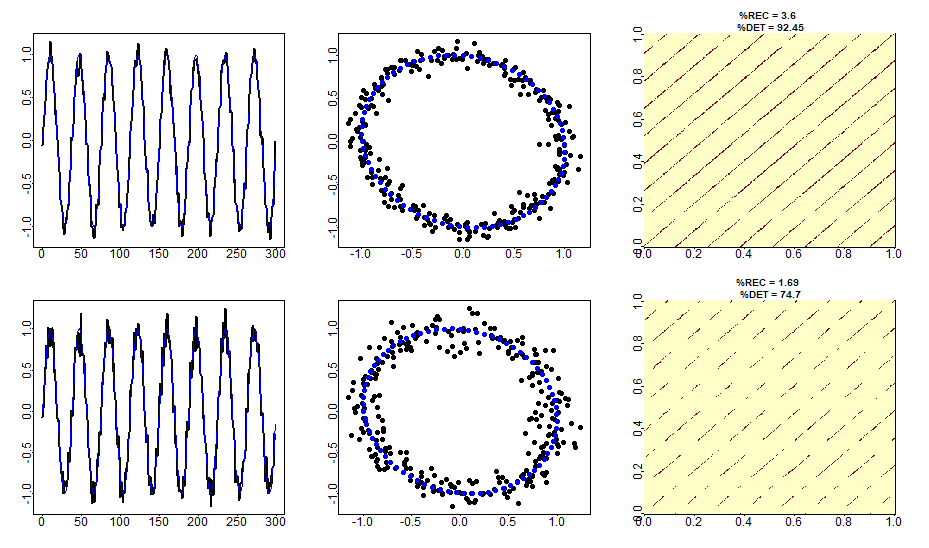}
\caption{The instability of percent recurrence and percent determinism when subject to added time-series noise.
For each reference signal $f_2(t) = \sin(2 \pi (3/300) t)$ (in blue), we define its noisy version (in black) as $f_1(t) = f_2(t)  + \mathcal{N}(0,\sigma \cdot \text{sd}(f_2(t)))$, where $\text{sd}(f_2(t))$ denotes the standard deviation of $f_2(t)$.
We vary $\sigma$ beteen  $0.1$ (top) and $\sigma = 0.15$ (bottom).
We select the embedding dimension ($M = 1$) and the time lag ($\tau = 10$) using the first minimum of the automutual information \cite{wallot2018analyzing}.
We fix the distance threshold to be 20\% of the maximum pairwise distance between the embeddings of $f_1$ and $f_2$.
We finally fix the minimum number of points to qualify a diagonal to $\minDL = 5$.}
\label{fig:instability_FT_DET}
\end{figure}

\section{Related Work}\label{sec:related work}

Given a set of point cloud data $X$ in Euclidean space, the Vietoris-Rips complex (Definition \ref{def:VR_cmplx}) of $X$ at radius $\epsilon$, denoted $\mathcal{R}_\epsilon(X)$, is a triangulation of $X$ in which any two of its points $x_i, x_j$ forms an edge $[x_i x_j]$ when $||x_i - x_j||\leq 2 \epsilon$, and any three edges form a triangle, any four triangles form a tetrahedron, and so on. 
The nested sequences of subcomplexes formed when we increase the filtration radius $\epsilon$ over time forms the Vietoris-Rips filtration of $X$ (Definition \ref{def:VR_filtrn}).
Given a continuous and tame real-valued function $f$, the sublevel-set of $f$ at level $\lambda$, denoted $\mathcal{S}_\lambda(f)$, is the set of all points in $f$ whose value is at most $\lambda$.
The nested sequence of sublevel-sets of $f$ formed as we increase $\lambda$ over time forms the sublevel-set filtration of $f$.
The birth and death radii of all holes in various dimensions that are formed during either Vietoris-Rips or sublevel-set filtration (i.e. the persistent homology of $X$ or $f$) can be summarized as ordered pairs in a persistence diagram (Definition \ref{def:dgm}) or equivalently as intervals in a persistence barcode (Definition \ref{def:bcd}).
One favorable result in persistent homology is its guaranteed theoretical stability to small data perturbations. 
That is, given two point clouds $X$ and $Y$ in Euclidean space, if the Hausdorff distance (Definition \ref{def:Hausdorf_dist}) between them is small, then the Bottleneck distance (Definition \ref{def:Bottleneck_dist}) between their persistence diagrams from Vietoris-Rips filtration will also be small \cite{CoEdHa2007,perea2015}.
As well, given two tame and continuous functions $f$ and $g$, if the infinity-norm between them is small, then the Bottleneck distance between their persistence diagrams from sublevel-set filtration will also be small \cite{cohen2007}.
These stability results are formally given below by 

\begin{equation}\label{eq:VR_stability_PH}
    \dB(\dgm(\mathcal{R}(X)),\dgm(\mathcal{R}(Y))) \leq 2 \dH(X,Y) \text{ and}
\end{equation}

\begin{equation}\label{eq:sublevelset_stability_PH}
    \dB(\dgm(\mathcal{S}(f)),\dgm(\mathcal{S}(g))) \leq 2 ||f-g||_\infty,
\end{equation}

\noindent This guaranteed stability in part motivates the increasing usage of persistent homology for analysis of time series.
In 2015, a persistence-based measure was introduced that quantifies how periodic a single univariate time series is \cite{perea2015}.
Termed the Sliding Windows and 1-Persistence Scoring Function (denoted $\score(f)$), this measure was later used to quantify periodicity in gene expressions data \cite{SW1Pers_published_}.
This measure works by first embedding a univariate time series $f$ via and embedding dimension and a time lag to obtain a sliding windows embedding, denoted by 

\[ \SW_{M,\tau}f(t) = (f(t),f(t+\tau),\dots,f(t+M\tau)), \]

\noindent where $M \in \vN$ and $\tau \in \vR$ denote the dimension and time lag. 
Vietoris-Rips filtration is then performed on this embedding to obtain a persistence diagram $\dgm(\SW_{M,\tau}f(T))$.
The normalized lifetime of the longest-surviving one dimensional hole, denoted $\mp(\dgm_1(\SW_{M,\tau}f(t)))$ during this filtration is then used to quantify how periodic $f$ is.
The authors of \cite{perea2015} ultimately are able to use (1) to prove stability of  $\score(f)$ under sufficiently close enough Fourier series approximations of $f$, denoted $S_Nf(t)$ (i.e. for sufficiently large enough degrees of truncation $N \in \vN$).
They are also able to prove that for sufficiently periodic time series data, the sliding windows embeddings will be roundest when the sliding window used has width proportional to the period of $f$. 
This latter finding could prove very useful when trying to determine the degree to which a frequency component appears in a signal with suspected seasonality (i.e. weather or climate data), which could in turn be useful for constructing such a persistence-based measure with guaranteed theoretical stability as a robust alternative to \%DET and \%REC. 
In the remainder of this paper, we introduce such a persistence-based measure, theoretical stability results for the same, and apply it on various synthetic series and real time series with suspected seasonality, ultimately showing that our measure consistently outperforms \%DET and \%REC. 

\section{Contributions and Organization}\label{sec:contributions_organization}

In this paper, we use the stability of persistent homology to construct a measure of time-series similarity with theoretical stability. 
We are further able to show its greater experimental stability in comparison with the commonly used measure of percent determinism on both synthetic and real data.
Our novel measure has the additional advantage of only requiring one parameter for its computation when compared to the four required of percent determinism.
Given a single time series $f_1$ and a reference sinusoidal $f_2$ with frequency $w_2$ Hertz, we first define a conditional periodicity score, $\score(f_1|f_2)$, which can be used to quantify the extent to which the frequency $w_2$ appears in $f_1$. 
We compare this measure to spectral magnitudes produced via Fourier analysis on synthetic (Figure \ref{fig:compare_condtional_score_stability_synthetic}) and real (Figure \ref{fig:freq_distn_comparison}) data.
We then extend this notion in order to measure the similarity between a pair of series $f_1$ and $f_2$, where $f_2$ need not be a pure sinusoidal. 
Given an embedding dimension $M \in \vN$, we define a conditional sliding window embedding of $f_1$ given $f_2$, which uses the estimated period of $f_2$ via Fourier analysis to determine the sliding window size when embedding $f_1$.
Our measure is then computed by normalizing lifetime of the longest-surviving 1-dimensional hole from Vietoris-Rips filtration on this conditional embedding, which produces the conditional periodicity score $\score(f_1|f_2)$. 
We then repeat this process to compute $\score(f_2|f_1)$.
Taking the average of both scores produces our novel similarity measure, the bi-conditional periodicity score $\score(f_1,f_2)$.
We are able to prove under which conditions $\score(f_1|f_2)$ reduces to $\score(f_1)$ (Proposition \ref{prop:conditional_score_approaches_score}), as well as the stability of $\score(f_1|f_2)$ to small time-series and frequency perturbations (Theorems \ref{thm:stability_conditional_score_change_f2} and \ref{thm:stability_conditional_score_change_f1}).
We then leverage these stability results to prove the stability of our bi-conditional periodicity score (Theorem \ref{thm:stability_biconditional_score}).
We are then able to prove the existence of a minimum embedding dimension which controls how close the conditional periodicity score is to its underlying value (Theorem \ref{thm:min_M_for_conditional_score}). 
The latter result suggests that sufficiently large dimensions may be required to acheive desired convergence of $\score(f_1|f_2)$, which motivates us to prove the stability of both $\score(f_1|f_2)$ and $\score(f_1,f_2)$ under dimension reduction (Theorem \ref{thm:stability_conditional_score_und_PCA} and Corollary \ref{cor:stability_biconditional_score_und_PCA}).
We then construct an algorithm for computing the conditional periodicity score (Algorithm \ref{alg:pseudocode}) and deduce its computational complexity.
Following this, we use our algorithm to show that our bi-conditional periodicity score is experimentally more stable under small time-series perturbations on both synthetic and real climate data when compared to percent determinism (Figures \ref{fig:comparing_stability_biconditional_score_REC} and \ref{fig:stability_results_climate_data}).

The remainder of this paper is organized as follows.
We begin with definitions to know regarding persistent homology, distance metrics for proving stability results, and our novel similarity measure (Section \ref{sec:def}).
We then provide our theoretical stability results in Section \ref{sec:stability}, followed by our experimental results in Section \ref{sec:exp_results}.
Lastly, we provide a discussion of our results and future directions we'd like to take this project in Section \ref{sec:discussion_future}.

\section{Definitions}\label{sec:def}

\subsection{Persistent Homology}

Here we introduce standard definitions of topological features used in this paper. 
More information can be found in \cite{Munkres1984,OtPoTiGrHa2017}.

\begin{definition}
    [Vietoris-Rips Complex]\label{def:VR_cmplx}
    Given a point cloud $X = \{x_1,x_2,\dots,x_n\} \in \vR^d$ and a filtration radius $\epsilon > 0$, the Vietoris-Rips complex of $X$ at $\epsilon$ is given by the collection of simplices below

    \[ \mathcal{R}_\epsilon(X) = \{ \sigma = [x_i \dots x_j] : i,j \in 1 \dots n; ||x_s - x_k|| \leq 2 \epsilon, s,k \in 1 \dots n \}. \]
\end{definition}

\begin{definition}
    [Vietoris-Rips Filtration]\label{def:VR_filtrn}
    Given a collection of increasing filtration radii $0 = \epsilon_1 < \epsilon_2 < \dots < \epsilon_r < \dots$ applied to a point cloud $X$, the collection of nested subcomplexes 

    \[ X \subset \mathcal{R}_{\epsilon_2}(X) \subset \dots \subset \mathcal{R}_{\epsilon_r}(X) \subset \dots \]

    \noindent denotes the Vietoris-Rips filtration of $X$, given by $\mathcal{R}(X)$.
\end{definition}



%


\begin{definition}
    [Persistence Diagram]\label{def:dgm}
    Given Vietoris-Rips filtration on a point cloud $X$, the collection of birth and death radii for each $p$-dimensional hole that appears can be summarized as ordered pairs in a persistence diagram given below by 

    \[ \dgm(\mathcal{R}(X)) = \{ (b_i,d_i):i=1,\dots,m_p,p=0,1,2,\dots\}. \]


    
    \noindent Unless otherwise stated, we simplify this notation to $\dgm(X)$, denoting $\dgm_p(X)$ as the collection of all ordered pairs for one dimension $p$ from $\dgm(X)$.
\end{definition}

\begin{definition}
    [Persistence Barcode]\label{def:bcd}
    Given Vietoris-Rips filtration on a point cloud $X$, the collection of birth and death radii for each $p$-dimensional hole that appears can be summarized as intervals in a persistence barcode given  below by

    \[ \bcd(\mathcal{R}(X)) = \{ [b_i,d_i):i=1,\dots,m_p,p=0,1,2,\dots\}. \]



    \noindent Unless otherwise stated, we simplify this notation to $\bcd(X)$, denoting $\bcd_p(X)$ as the collection of all intervals for one dimension $p$ from $\bcd(X)$.
\end{definition}

\subsection{Distance Metrics}

Here we introduce standard definitions of distances used in this paper.
See, for instance, the book by Burago, Bugaro, and Ivanov~\cite{BuBuIv2001} for details. 



\begin{definition}[Hausdorff Distance]\label{def:Hausdorf_dist}
  Given two sets of points $X$ and $Y$ in a common metric space, the Hausdorff distance between them is given by 
  \[\dH(X,Y) = \inf{\{\epsilon>0:X \subseteq Y^{\epsilon}, Y \subseteq X^{\epsilon}\}}, ~\text{ where }
  X^{\epsilon} = \underset{x \in X}{\bigcup} B_{\epsilon}(x) \text{ and } Y^{\epsilon}= \underset{y \in Y}{\bigcup} B_{\epsilon}(y)\]
  denote the union of all $\epsilon$-balls centered at each point in either set. 
\end{definition}

\begin{definition}[Hausdorff Definition of Gromov-Hausdorff Distance]\label{def:Gromov_Hausdorf_dist_Hausdorf}
  Given two sets of points $X$ and $Y$, the Gromov-Hausdorff distance between them is given by 
  \[ \dGH(X,Y) = \inf{\{\dH(f(X),g(Y)) : f:X\rightarrow S, g:Y\rightarrow S \}}, \]
  where $f$ and $g$ are isometric embeddings of $X$ and $Y$ into a common metric space $S$.
  If $X$ and $Y$ lay in a shared metric space $S$,  then $\dGH(X,Y) \leq \dH(X,Y)$ \cite{adams2024}.
\end{definition}

\begin{definition}[Distortion Definition of Gromov-Hausdorff Distance]\label{def:Gromov_Hausdorf_dist_distortion}
  An alternative definition of the Gromov-Hausdorff distance between two sets of points $X$ and $Y$ is given by 
  \[ \dGH(X,Y) = \frac{1}{2} \inf{\{\dis(R):R:X\rightarrow Y} \in \mathcal{R}(X,Y)\}, \]
  where $R$ is a relation between $X$ and $Y$ whose distortion is defined by 
  \[ \dis(R) = \sup \{|d_X(x,x') - d_Y(y,y')| \,:\, (x,y), (x',y') \in R\},\]
  where $d_{X}$ and $d_{Y}$ are the corresponding metrics for $X$ and $Y$, respectively. 
\end{definition}

\begin{definition}[Bottleneck Distance]\label{def:Bottleneck_dist}
  Given two finite sets of points $X$ and $Y$, let $\dgm(X)$ and $\dgm(Y)$ denote the persistence diagrams of a chosen dimension obtained from the Vietoris-Rips (VR) filtration on $X$ and $Y$, respectively.
  Then the Bottleneck distance between $\dgm(X)$ and $\dgm(Y)$ is given by 
  
  \[ \dB(\dgm(X),\dgm(Y)) = \underset{\phi}{\inf}\,\underset{x}{\sup} \norm{x - \phi(x)}_{\infty}, \]
  
  \noindent where $\phi:\dgm(X) \rightarrow \dgm(Y)$ denotes a bijection between $\dgm(X)$ and $\dgm(Y)$, including points along the diagonal in either diagram when they both do not share the same cardinality.
  
\end{definition}

\subsection{The Conditional Periodicity Scoring Function}

Here we denote the definitions for our new scoring function.

\begin{definition}
    [Conditional Sliding Windows Embedding]\label{def:conditional_SWE}
    Given a continuous univariate time series $f_1:T \subseteq \vR \rightarrow \vR$ and a reference series $f_2$ with a frequency of $w_2 = \arg\max_w \left\vert \int_T f_2(t) e^{i 2 \pi w t}dt \right\vert$ Hertz (Hz) (i.e. cycles per second), define the time lag $\tau_{w_2} = \frac{1}{w_2(M+1)}$ for embedding dimension $M \in \vN$. 
    Then the conditional sliding windows embedding of $f_1$ given $f_2$ is given by 

    \[ \SW_{M,\tau_{w_2}} f_{1|2}(t) = (f_1(t),f_1(t+\tau_{w_2}), \dots, f_1(t+M\tau_{w_2})). \]
\end{definition}

\begin{definition}
    [Conditional Periodicity Score]\label{def:conditional_score}
    Denote the lifetime of the longest-surviving 1-dimensional hole from Vietoris-Rips filtration on the conditional sliding windows embedding of a continuous time series $f_1$ given a reference series $f_2$ with frequency $w_2 = \arg\max_w \left\vert \int_T f_2(t) e^{i 2\pi w t}dt \right\vert$ Hz by $\mp(\dgm_1(\SW_{M,\tau_{w_2}}f_{1|2}(T)))$.
    Then the conditional periodicity score of $f_1$ given $f_2$ is given by the normalization of this lifetime below:

    \[ \score(f_1|f_2) = \frac{\mp(\dgm_1(\SW_{M,\tau_{w_2}}f_{1|2}(T)))}{\sqrt{3}}. \]
\end{definition}

\begin{definition}
    [Bi-Conditional Periodicity Score]\label{def:bi_conditional_score}
    Given a continuous univariate time series $f_1 : T_1 \subseteq \vR \rightarrow \vR$ and a reference series $f_{2} : T_2 \subseteq \vR \rightarrow \vR$, denote the frequency of each signal as $w_1 = \arg\max_w \left\vert \int_{T_1}f_1(t) e^{i 2 \pi  w t} dt \right\vert$ and $w_2 = \arg\max_w \left\vert \int_{T_2}f_2(t) e^{i 2 \pi  w t} dt \right\vert$.
    Given time lags $\tau_{w_1} = \frac{1}{w_1(M+1)}$ and $\tau_{w_2} = \frac{1}{w_2(M+1)}$, the bi-conditional periodicity score of $f_1$ and $f_2$ is given by 

    \[ \score(f_1,f_2) = \frac{\score(f_1|f_2) + \score(f_2|f_1)}{2}. \]
\end{definition}

\section{Theoretical Results}\label{sec:stability}

\subsection{Reduction to the Periodicity Score}

Here we discuss a result that clarifies when our conditional periodicity score $\score(f_1|f_2)$ simplifies to the periodicity score of $f_1$ as defined in \cite{perea2015,SW1Pers_published_}.
Specifically, as $\tau_{w_2}$ approaches $\tau_{w_1} = \frac{1}{w_1(M+1)}$, then $\SW_{M,\tau_{w_2}}f_{1|2}(t)$ approaches $\SW_{M,\tau_{w_1}}f_1(t)$ and hence $\score(f_1|f_2)$ approaches $\score(f_1)$. 
We formalize this notion in Proposition \ref{prop:conditional_score_approaches_score} below.

\begin{proposition}[Reduction to Periodicity Score]\label{prop:conditional_score_approaches_score}
Let $f_1,f_2 : \vR \subseteq \vR \rightarrow \vR$ be two continuous and differentiable univariate time series in which $f_1$ has frequency $w_1 = \arg\max_w \left\vert \int_\vR f_1(t)e^{i 2\pi w t}dt \right\vert$ and $f_2$ is a reference sinusoidal with frequency $w_2$.
    Then

\[ \lim_{\substack{\frac{1}{w_2} \rightarrow \frac{1}{w_1}}} \SW_{M,\tau_{w_2}}f_{1|2}(t) = \SW_{M,\tau_{w1}}f_1(t). \]
\end{proposition}

\begin{proof}
  Let $i \in \{1,2,\ldots,M+1\}$ for fixed $M \in \mathbb{N}$.
  Since $f_1$ is continuous and differentiable on $\vR$, it is also continuous and differentiable on any subset of $\vR$.
  Consider the subinterval 
  $I_i=\left[ t + (i-1) \tau_{w_2}, t + (i-1) \tau_{w_1} \right]$.
  Then $f_1$ is also continuous and differentiable on $I_i$.
  By the Mean Value Theorem (MVT), there exists some 
  %
  $c_i \in I_i$ such that 
  \[ f'_1(c_i)\hspace*{-0.03in}\left[\left(t+\frac{(i-1)}{w_1(M+1)}\right)
    \text{\hspace*{-0.03in}$-$\hspace*{-0.03in}}
    \left(t+\frac{(i-1)}{w_2(M+1)}\right)\right]
    \text{\hspace*{-0.03in}$=$\hspace*{0.0in}}
    f_1\hspace*{-0.03in}\left(t+\frac{(i-1)}{w_1(M+1)}\right)-
    f_1\hspace*{-0.03in}\left(t+\frac{(i-1)}{w_2(M+1)}\right)\hspace*{-0.02in}.\]

    Then we get the following equality:
    \begin{align*}
      \norm{\SW_{M,\tau_{w_2}}f_{1|2}(t) - \SW_{M,\tau_{w_1}}f_1(t)}_{2} \hspace*{3.5in}\\
      = \left(\sum_{i=1}^{M+1} \left\vert f_1\left(t+\frac{(i-1)}{w_2(M+1)}\right)-f_1\left(t+\frac{(i-1)}{w_1(M+1)}\right)\right\vert^{2}\right)^{\frac{1}{2}} \hspace*{0.23in} \\
      = \left( \sum_{i=1}^{M+1} \vert f_1'(c_i) \vert^2  \left\vert \left(t+\frac{(i-1)}{w_2(M+1)}\right) -\left(t+\frac{(i-1)}{w_1(M+1)}\right) \right\vert^{2}\right)^{\frac{1}{2}} \\
      = \left( \sum_{i=2}^{M+1} \vert f_1'(c_i) \vert^2 \left\vert \frac{i-1}{M+1} \right\vert^2 \left\vert \frac{1}{w_1} - \frac{1}{w_2} \right\vert^2 \right)^{\frac{1}{2}}. \hspace*{1.4in} \qedhere
    \end{align*}

\end{proof}
A direct consequence of Proposition \ref{prop:conditional_score_approaches_score} is that $\score(f_1 | f_2) \rightarrow \score(f_1)$ as $\frac{1}{w_2} \rightarrow \frac{1}{w_1}$.

\subsection{Stability Under Time Series and Frequency Perturbations}
We first prove that given a fixed time series $f_1$, small variations in the frequency of a reference sinusoidal $f_2$ guarantee small changes in the conditional periodicity score of $f_1$ given $f_2$, which intuitively quantifies the extent to which the frequency of $f_2$ is present within $f_1$. 
We prove this result formally in Theorem \ref{thm:stability_conditional_score_change_f2} below. 
We next show the opposite in Theorem \ref{thm:stability_conditional_score_change_f1}. 
That is, we show that under small perturbations of the series $f_1$ the conditional periodicity score of $f_1$ given a reference sinusoidal $f_2$ at a fixed frequency of $w_2$ also undergoes small changes.
We then combine these results to prove the stability of the bi-conditional periodicity score in Theorem \ref{thm:stability_biconditional_score}. 

\begin{theorem}
    [Stability of the Conditional Periodicity Score Under Changes in Reference Frequency]\label{thm:stability_conditional_score_change_f2}
    Let $f_1,f_{21},f_{22} : \vR \subseteq \vR \rightarrow \vR$ be three continuous and differentiable univariate time series with frequencies $w_1 = \arg\max_w \left\vert \int_T f_1(t)e^{i 2\pi w t}dt \right\vert$, where $f_{21}$ and $f_{22}$ are reference sinusoidals with frequencies $w_{21}$, and $w_{22}$.
    Then
    
    \[ |\score(f_1|f_{21}) - \score(f_1|f_{22})| \leq \frac{4}{\sqrt{3}} \sqrt{\frac{M}{M+1}} \left\vert \frac{1}{w_{21}} - \frac{1}{w_{22}} \right\vert \sqrt{\sum_{i=1}^M |f'_1(c_i)|^2}. \]
\end{theorem}

\begin{proof}
Without loss of generality, suppose $w_{21} < w_{22}$. Let $M \in \vN$.
Assume $c_i \in I_i = (t+i\tau_{w_{22}},t+i\tau_{w_{21}})$ for $i = 1,2,\dots,M$.
Then

\begin{align*}
    ||\SW_{M,\tau_{w_{21}}}f_{1|21}(t) - \SW_{M,\tau_{w_{22}}}f_{1|22}(t)||^2 & = \sum_{i=1}^M |f_1(t + i\tau_{w_{21}}) - f_1(t + i\tau_{w_{22}})|^2 \\
    & = \sum_{i=1}^M |f'_1(c_i)|^2 |(t+i\tau_{w_{21}}) - (t+i\tau_{w_{22}})|^2 \text{ by Mean Value Theorem} \\
    & = \frac{M^2}{M+1} \left\vert \frac{1}{w_{21}} - \frac{1}{w_{22}} \right\vert^2 \sum_{I=1}^M |f'_1(c_i)|^2.
\end{align*}

\noindent Then $||\SW_{M,\tau_{w_{21}}}f_{1|21}(t) - \SW_{M,\tau_{w_{22}}}f_{1|22}(t)|| \leq \sqrt{\frac{M}{M+1}} \left\vert \frac{1}{w_{21}} - \frac{1}{w_{22}} \right\vert \sqrt{\sum_{I=1}^M |f'_1(c_i)|^2} \quad (\ast)$.
Let $\epsilon > (\ast)$.
Then by definition, 

\begin{align*}
    \dH \left( \SW_{M,\tau_{w_{21}}}f_{1|21}(T),\SW_{M,\tau_{w_{22}}}f_{1|22}(T) \right) & = \inf \{ \gamma > 0 : ||\SW_{M,\tau_{w_{21}}}f_{1|21}(t) - \SW_{M,\tau_{w_{22}}}f_{1|22}(T)||, \\
    & ||\SW_{M,\tau_{w_{22}}}f_{1|22}(t) - \SW_{M,\tau_{w_{21}}}f_{1|21}(T)|| \leq \gamma \text{, } \forall t \in T \} \\
    & \leq (\ast).
\end{align*}

\noindent Then for $T \subseteq \vR$, by Equation \ref{eq:VR_stability_PH}, $\dB \left( \dgm_1(\SW_{M,\tau_{w_{21}}}f_{1|21}(T)),\dgm_1(\SW_{M,\tau_{w_{22}}}f_{1|22}(T)) \right) \leq 2 (\ast)$.
Let $(b_{21}^{\max},d_{21}^{\max})$ and $(b_{22}^{\max},d_{22}^{\max})$ be the birth and death radii of the longest-surviving 1-dimensional holes from Vietoris-Rips filtration on $\SW_{M,\tau_{w_{21}}}f_{1|21}(T)$ and $\SW_{M,\tau_{w_{22}}}f_{1|22}(T)$.
Then 

\begin{align*}
    \left\vert \mp \left( \dgm_1 \left( \SW_{M,\tau_{w_{21}}}f_{1|21}(T) \right) \right) - \mp \left( \dgm_1 \left( \SW_{M,\tau_{w_{22}}}f_{1|22}(T) \right) \right) \right\vert & = \left\vert (d_{21}^{\max} - b_{21}^{\max}) - (d_{22}^{\max} - b_{22}^{\max}) \right\vert \\
    & \leq |d_{21}^{\max} - d_{22}^{\max}| + |b_{22}^{\max} - b_{21}^{\max}|.
\end{align*}

\noindent Let $\phi: \dgm_1 \left( \SW_{M,\tau_{w_{21}}}f_{1|21}(T) \right) \rightarrow \dgm_1 \left( \SW_{M,\tau_{w_{22}}}f_{1|22}(T) \right)$ be an injection given by $(b_{21}^i,d_{21}^i) \rightarrow (b_{22}^j,d_{22}^j)$.
Then

\begin{align*}
    \dB \left( \dgm_1 \left( \SW_{M,\tau_{w_{21}}}f_{1|21}(T) \right), \dgm_1 \left( \SW_{M,\tau_{w_{22}}}f_{1|22}(T) \right) \right) & = \inf_\phi \sup_i ||(b_{21}^i,d_{21}^i) - \phi((b_{21}^i,d_{21}^i))||_\infty \\
    & = \inf_\phi \sup_i \max \left\{ |b_{21}^i - b_{22}^j|, |d_{21}^i - d_{22}^j| \right\} \\
    & \geq |b_{21}^{\max} - b_{22}^{\max}|, |d_{21}^{\max} - d_{22}^{\max}|.
\end{align*}

\noindent Hence 

\begin{align*}
    |\score(f_1|f_{21}) - \score(f_1|f_{22})| & = \left\vert \mp \left( \dgm_1 \left( \SW_{M,\tau_{w_{21}}}f_{1|21}(T) \right) \right) - \mp \left( \dgm_1 \left( \SW_{M,\tau_{w_{22}}}f_{1|22}(T) \right) \right) \right\vert \\
    & \leq 2 \dB \left( \dgm_1 \left( \SW_{M,\tau_{w_{21}}}f_{1|21}(T) \right), \dgm_1 \left( \SW_{M,\tau_{w_{22}}}f_{1|22}(T) \right) \right) \\
    & \leq 4 \dH \left( \SW_{M,\tau_{w_{21}}}f_{1|21}(T),\SW_{M,\tau_{w_{22}}}f_{1|22}(T) \right) \\
    & \leq \frac{4}{\sqrt{3}} \sqrt{\frac{M}{M+1}} \left\vert \frac{1}{w_{21}} - \frac{1}{w_{22}} \right\vert \sqrt{\sum_{i=1}^M |f'_1(c_i)|^2}.
\end{align*}
\end{proof}

\begin{theorem}
    [Stability of Conditional Periodicity Score Under Changes in Time Series]\label{thm:stability_conditional_score_change_f1}
    Let $f_{11},f_{12} : \vR \rightarrow \vR$ be two continuous and differentiable univariate time series and $f_2$ be a reference sinusoidal with frequency $w_2$ Hz.
    Then 

    \[ |\score(f_{11}|f_2) - \score(f_{12}|f_2)| \leq  4 ||f_{11}-f_{12}||_\infty \sqrt{\frac{M}{3}} . \]
\end{theorem}

\begin{proof}

Let $i = 1, 2, \dots, M$.
Because for all $t \in \vR$, $t + i \tau_{w_2}$ is also a real number, we have the following 

\begin{align*}
    || \SW_{M,\tau_{w_2}}f_{11|2}(t) - \SW_{M,\tau_{w_2}}f_{12|2}(t)||^2 & = \sum_{i=1}^M | f_{11}(t+i\tau_{w_2}) - f_{12}(t+i\tau_{w_2})|^2 \\
    & \leq \sum_{i=1}^M \left( \max_{t \in \vR} \Big\{ |f_{11}(t+i\tau_{w_2}) - f_{12}(t+i\tau_{w_2})| \Big\} \right)^2 \\
    & = \sum_{i=1}^M ||f_{11} - f_{12}||^2_\infty \\
    & = M || f_{11} - f_{12} ||^2_\infty.
\end{align*}

\noindent Then $||\SW_{M,\tau_{w_2}}f_{11|2}(t) - \SW_{M,\tau_{w_2}}f_{12|2}(t)|| \leq \sqrt{M}||f_{11}-f_{12}||_\infty$.
Then by similar arguments as in the proof of Theorem \ref{thm:stability_conditional_score_change_f2}, for some $T \subseteq 
\vR$ we have that 

\begin{align*}
    & \dH(\SW_{M,\tau_{w_2}}f_{11|2}(T), \SW_{M,\tau_{w_2}}f_{12|2}(T)) \leq \sqrt{M} ||f_{11} - f_{12}||_\infty \text{ and} \\
    & \left\vert \mp \left( \dgm_1(\SW_{M,\tau_{w_2}}f_{11|2}(T)) \right) - \mp \left( \dgm_1(\SW_{M,\tau_{w_2}}f_{12|2}(T)) \right) \right\vert \\
    & \hspace{5cm} \leq 2 \dB \left( \dgm_1(\SW_{M,\tau_{w_2}}f_{11|2}(T)) , \dgm_1(\SW_{M,\tau_{w_2}}f_{12|2}(T)) \right)
\end{align*}

\noindent Therefore, by Equation \ref{eq:VR_stability_PH} we have that 

\begin{align*}
    | \score(f_{11}|f_2) - \score(f_{12}|f_{2})| & = \frac{1}{\sqrt{3}} \left\vert \mp \left( \dgm_1(\SW_{M,\tau_{w_2}}f_{11|2}(T)) \right) - \mp \left( \dgm_1(\SW_{M,\tau_{w_2}}f_{12|2}(T)) \right) \right\vert \\
    & \leq \frac{2}{\sqrt{3}} \dB \left( \dgm_1(\SW_{M,\tau_{w_2}}f_{11|2}(T)) , \dgm_1(\SW_{M,\tau_{w_2}}f_{12|2}(T)) \right) \\
    & \leq \frac{4}{\sqrt{3}} \dH(\SW_{M,\tau_{w_2}}f_{11|2}(T), \SW_{M,\tau_{w_2}}f_{12|2}(T)) \\
    & \leq 4 ||f_{11}-f_{12}||_\infty \sqrt{\frac{M}{3}}.
\end{align*}
\end{proof}

\begin{theorem}
    [Stability of the Bi-Conditional Periodicity Score]\label{thm:stability_biconditional_score}
    Let $f_1,f_{11},f_{12},f_2,f_{21},f_{22} : \vR \rightarrow \vR$ be three continuous and differentiable univariate time series with frequencies given by $w_k = \arg\max_w \left\vert \int_{\vR} f_k(t) e^{i 2\pi w t}dt \right\vert$ for $k=1,11,12,2,21,22$.
    Without loss of generality, suppose that $w_{21} < w_{22}$ and $w_{11} < w_{12}$.
    Then given some $c_i \in I_i = (t+i\tau_{w_{22}},t+i\tau_{w_{21}})$ and $d_i \in J_i = (t+i\tau_{w_{12}},t+i\tau_{w_{11}})$, we have that

    \[ |\score(f_1,f_{21}) - \score(f_1,f_{22})| \leq  \frac{2}{\sqrt{3}} \left( \sqrt{\frac{M}{M+1}} \left\vert \frac{1}{w_{21}} - \frac{1}{w_{22}} \right\vert \sqrt{\sum_{i=1}^M |f_1'(c_i)|^2} + ||f_{21} - f_{22}||_\infty \sqrt{M} \right) \text{ and} \]

    \[ |\score(f_{11},f_2) - \score(f_{12},f_2)| \leq \frac{2}{\sqrt{3}} \left( \sqrt{\frac{M}{M+1}} \left\vert \frac{1}{w_{11}} - \frac{1}{w_{12}} \right\vert \sqrt{\sum_{i=1}^M |f_2'(d_i)|^2} + \sqrt{M} ||f_{11} - f_{12}||_\infty  \right). \]
\end{theorem}

\begin{proof}
By Theorem \ref{thm:stability_conditional_score_change_f2}, 

\[ |\score(f_1|f_{21}) - \score(f_1|f_{22})| \leq \frac{4}{\sqrt{3}} \sqrt{\frac{M}{M+1}} \left\vert \frac{1}{w_{21}} - \frac{1}{w_{22}} \right\vert \sqrt{\sum_{i=1}^M|f'_1(c_i)|^2} \quad \text{ and} \]
\[ |\score(f_2|f_{11}) - \score(f_2|f_{12})| \leq \frac{4}{\sqrt{3}} \sqrt{\frac{M}{M+1}} \left\vert \frac{1}{w_{11}} - \frac{1}{w_{12}} \right\vert \sqrt{\sum_{i=1}^M|f'_2(d_i)|^2}. \]

\noindent By Theorem \ref{thm:stability_conditional_score_change_f1}, 

\[ |\score(f_{21}|f_1) - \score(f_{22}|f_1)| \leq 4 ||f_{21} - f_{22}||_\infty \sqrt{\frac{M}{3}} \quad \text{ and} \]
\[ |\score(f_{11}|f_2) - \score(f_{12}|f_2)| \leq 4 ||f_{11} - f_{12}||_\infty \sqrt{\frac{M}{3}}. \]

\noindent Then we have that 

\begin{align*}
    |\score(f_1,f_{21}) - \score(f_1,f_{22})| & = \frac{1}{2} \left\vert \left( \score(f_1|f_{21}) + \score(f_{21}|f_1) \right) - \left( \score(f_1|f_{22}) + \score(f_{22}|f_1) \right) \right\vert \\
    & \leq \frac{1}{2} \left( |\score(f_1|f_{21}) - \score(f_1|f_{22})| + |\score(f_{21}|f_1) - \score(f_{22}|f_1)| \right) \\
    & \leq \frac{1}{2} \left( \frac{4}{\sqrt{3}} \sqrt{\frac{M}{M+1}} \left\vert \frac{1}{w_{21}} - \frac{1}{w_{22}} \right\vert \sqrt{\sum_{i=1}^M f'_1(c_i)|^2} + 4 ||f_{21} - f_{22}||_\infty \sqrt{\frac{M}{3}} \right). 
\end{align*}

\noindent A similar argument proves that 

\[ |\score(f_{11},f_2) - \score(f_{12},f_2)| \leq \frac{1}{2} \left( 4||f_{11}-f_{12}||_\infty \sqrt{\frac{M}{3}} + \frac{4}{\sqrt{3}}\sqrt{\frac{M}{M+1}} \left\vert \frac{1}{w_{11}} - \frac{1}{w_{12}} \right\vert \sqrt{\sum_{i=1}^M |f_2'(d_i)|^2}  \right). \]
\end{proof}

\subsection{A Minimum Embedding Dimension}

For the conditional periodicity score, the only parameter required is the embedding dimension $M \in \vN$.
We are hence motivated to identify a minimum dimension required to optimize the performance of $\score(f_1|f_2)$, which we formalize in Theorem \ref{thm:min_M_for_conditional_score} below.

\begin{theorem}
[A Minimum Embedding Dimension]\label{thm:min_M_for_conditional_score}

Let $\epsilon > 0$.
Suppose $f_1,f_2 : \vR \rightarrow \vR$ are continuous and differentiable univariate time series and the frequency of $f_2$ is given by $w_2 = \arg\max_w \left\vert \int_\vR f_2(t) e^{i 2\pi w t}dt \right\vert$.
Then, any embedding dimension $M_2 > M_1 \geq \mathcal{M}$ $\in$ $\mathbb{N}$ for $\mathcal{M} = \left\lceil{ \frac{1}{w_2\epsilon} }\right\rceil$ guarantees that the change in respective conditional periodicity scores $\score_{M_2}(f_1|f_2)$ and $\score_{M_1}(f_1|f_2)$ are linearly bounded above by $\epsilon$.
\end{theorem}

\begin{proof} 
We first show that $\tau(M) = \frac{1}{w_2(M+1)}$ is a Cauchy sequence of $M$. 
By the definition of $\mathcal{M}$, $\mathcal{M} \geq \frac{1}{w_2\epsilon}$ and hence $\frac{1}{w_{2}\mathcal{M}} \leq \epsilon$ $\implies$ $\frac{1}{w_2} \left(\frac{1}{\mathcal{M}}\right) \leq \epsilon$. 
Since $M_2 > M_1 \geq \mathcal{M}$, $M_2+1 > M_1 + 1 \geq \mathcal{M}+1$ and hence $\frac{1}{M_2+1} < \frac{1}{M_1+1} \leq \frac{1}{\mathcal{M}+1} < \frac{1}{\mathcal{M}}$. 
Thus we get

\begin{align}
  \vert \tau(M_1) - \tau(M_2) \vert & = \frac{1}{w_{2}}\left\vert \frac{1}{M_1+1} - \frac1{M_{2}+1}\right\vert \nonumber \\
  & = \frac{1}{w_2} \left(\frac{1}{M_1+1} - \frac{1}{M_2+1} \right) \nonumber \\
  & \leq \frac{1}{w_2} \left(\frac{1}{\mathcal{M}+1} - \frac{1}{M_2+1} \right) \nonumber \\
  & < \frac{1}{w_2} \left(\frac{1}{\mathcal{M}} \right) \leq \epsilon. \label{eq:tauChy}
\end{align}

\noindent Define the zero-padded conditional sliding windows embedding of $f_1$ given $f_2$ for the smaller embedding dimension $M_1$ by  $\SW_{M_1,\tau(M_1)}f_{1|2}(t) = \Big( f_1(t), \ldots, f_1(t+M_1\tau(M_1)), \underbrace{0, \ldots, 0}_{M_2 - M_1} \Big)^T$.
Then for $i=1, \dots, M_1+1$, by the MVT there exists some $c_i \in I_i = \Big( t + (i-1)\tau(M_2), t + (i-1)\tau(M_1) \Big)$ such that 

\begin{align}
  |f_1'(c_i)| \Big\vert (i-1) \Big( \tau(M_1) - \tau(M_2) \Big) \Big\vert = \Big\vert f_1\big(t+(i-1)\tau(M_1)\big) - f_1\big(t+(i-1)\tau(M_2)\big) \Big\vert. \label{eq:tauMVT} \\
\end{align}

\noindent Applying the results in Equations (\ref{eq:tauChy}) and (\ref{eq:tauMVT}), we obtain the following relation.
\begin{align*}
  \norm{ \SW_{M_1,\tau(M_1)} f_{1|2}(t) - \SW_{M_2,\tau(M_2)} f_{1|2}(t) }_2^2  \hspace*{3.5in} \\
  \hspace*{0.2in} = \sum_{i=1}^{M_1+1} \Big\vert f_1\Big(t+(i-1)\tau(M_1)\Big) - f_1\Big(t+(i-1)\tau(M_2)\Big) \Big\vert^2
  + \sum_{M_1+1}^{M_2+1} \Big\vert f_1\Big(t+(i-1)\tau(M_2)\Big) \Big\vert^2 \\
  = \sum_{i=1}^{M_1+1} |f_1'(c_i)|^2 |i-1|^2 |\tau(M_1)-\tau(M_2)|^2 ~~+~~
  \sum_{M_1+1}^{M_2+1} \Big\vert f_1\Big(t+(i-1)\tau(M_2)\Big)\Big\vert^2 \hspace*{0.95in} \\
   \leq \epsilon^2 (M_1+1)^3 \sqrt{\sum_{i=1}^{M_1+1} |f_1'(c_i)|^2} ~~+~~ \sum_{M_1+1}^{M_2+1} \Big\vert f_1\Big(t+(i-1)\tau(M_2)\Big)\Big\vert^2 \hspace*{1.7in} \\
   = \epsilon^2 \cdot g(M_1,f_1) ~~+~~ h(M_1,M_2,f_1) \hspace*{3.78in} \\
   \leq \left(\epsilon \cdot \sqrt{g(M_1,f_1)} + \sqrt{h(M_1,M_2,f_1)}\right)^2, \hspace*{3.35in} 
\end{align*}
since $g(M_1,f_1)$ and $h(M_1,M_2,f_1)$ are nonnegative constants for any given comparison, and $\epsilon > 0$.

\[\text{Then } \dH \hspace*{-0.01in} \Big(\hspace*{-0.01in} \SW_{M_1,\tau(M_1)}f_{1|2}(t) , \SW_{M_2,\tau(M_2)}f_{1|2}(t) \hspace*{-0.01in} \Big) \hspace*{-0.01in} \leq \epsilon \cdot \sqrt{g(M_1,f_1)} + \sqrt{h(M_1,M_2,f_1)}\,, \hspace*{0.5in}\]
which implies that 
\[ \dB \hspace*{-0.03in}\Big(\hspace*{-0.03in} \dgm_1(\SW_{M_1,\tau(M_1)}f_{1|2}(t)),\dgm_1(\SW_{M_2,\tau(M_2)}f_{1|2}(t))\hspace*{-0.03in}\Big) \hspace*{-0.02in} \leq \epsilon \cdot 2 \sqrt{g(M_1,f_1)} + 2\sqrt{h(M_1,M_2,f_1)}. \]

By a similar argument as in the proof of Theorem \ref{thm:stability_conditional_score_change_f2}, this means that 

\begin{align*}
  \Big\vert \mp \Big( \dgm_1(\SW_{M_1,\tau(M_1)}f_{1|2}(t)) \Big) - \mp \Big( \dgm_1(\SW_{M_2,\tau(M_2)}f_{1|2}(t)) \Big) \Big\vert \hspace*{1.4in}\\
  \leq \epsilon \cdot 4 \sqrt{g(M_1,f_1)} + 4\sqrt{h(M_1,M_2,f_1)} \hspace*{0.4in} \\
  \implies \hspace*{0.4in} \vert \score_{M_1}(f_1|f_2) - \score_{M_2}(f_1|f_2) \vert \leq \frac{4}{\sqrt{3}} \left( \epsilon \cdot \sqrt{g(M_1,f_1)} + \sqrt{h(M_1,M_2,f_1)}\right)\,.
\end{align*}

Hence the difference between the scores under $M_1$ and $M_2$ is bounded above by a linear function of $\epsilon$ with positive slope. 
Therefore, the smaller the $\epsilon$-precision is, the closer these scores are to each other. 
\end{proof}

\subsection{Stability Under Dimension Reduction}

As shown by Theorem \ref{thm:min_M_for_conditional_score}, for sufficiently high enough precision of our observed score $\score_M(f_1|f_2)$ to the true underlying score $\score(f_1|f_2)$ (i.e. for sufficiently small enough $\epsilon > 0$), larger embedding dimensions are required.
For sufficiently large embedding dimensions, pairwise distances between points have been shown to concentrate \cite{AgHiKe2001}.
As well, the bottleneck when using the periodicity score in practice is the computation of the dimension 1 persistence diagram of the Vietoris-Rips filtration of the SWE.
In general, for a point cloud $X$ with $N$ points, the computation of $\dgm_1(X)$ runs in $O(N^6)$ time, although faster approaches may be available in lower dimensions \cite{KoMeRoTu2024}.
Both of these limitations in the high-dimensional setting inspire us to prove the stability of our scoring function under dimension reduction of the conditional sliding windows embedding.
That is, given a sufficiently large enough dimension $M$ to reach this desired precision, so long as the top principle components capture a majority of the variance in the conditional embedding upon orthogonal projection, the conditional periodicity score won't change much.
We hence study the stability of our conditional and bi-conditional periodicity scores under principal component analysis (PCA), a widely used dimension reduction technique \cite{Jo2002}.

\begin{theorem}[Stability of Conditional Periodicity Score Under Dimension Reduction] \label{thm:stability_conditional_score_und_PCA}
    Let $K \leq M+1$ for $K \in \vN$. Suppose $f_1,f_2 : \vR \rightarrow \vR$ are continuous and differentiable univariate time series. 
    For $T \subseteq \vR$, define the orthogonal projection of $X = \SW_{M,\tau_{w_2}}f_{1|2}(T)$ onto its top $K$ principal components by $\phi(X) = ( \langle \vc_1,X \rangle,\ldots, \langle \vc_K,X \rangle )^T$ for orthonormal eigenvectors and corresponding eigenvalues $\{ \vc_k,\lambda_k \}_{k=1}^N$ produced by PCA. Suppose $X$ contains $N \in \vN$ points. Denote the conditional periodicity score under $\phi$, $\score_\phi(f_1 \vert f_2)$, as the normalized maximum 1-d persistence from VR-filtration on $\phi(X)$. 
    Then 
    \[ | \score(f_1 | f_2) - \score_\phi(f_1 | f_2) | \leq \sqrt{\frac{8}{3}} \sqrt[4]{\sum_{i=K+1}^N \lambda_i^2}.\]
\end{theorem}

\begin{proof}
  Our proof is inspired by  methods used in \cite{MaKrGa2024}.
  Notice that $\phi$ is a relation in $\mathcal{R}(X,Y)$ for $Y = \phi(X)$.
  Then by Definition \ref{def:Gromov_Hausdorf_dist_distortion} we have that $\dGH(X,Y) \leq \frac{1}{2} \dis(\phi)$, where $\dis(\phi)^{2} = \norm{D_X - D_Y}_{\max}^{2}$ with $D_X$ being the matrix of pairwise distances in $X$.
  We get that \cite[Lemma 3.9]{MaKrGa2024}
    \[\norm{D_X - D_Y}_{\max}^2 \leq \norm{D_X^{\circ^2} - D_Y^{\circ^2}}_{\max},\]
    where $D_X^{\circ^2}$ denotes the matrix of squared pairwise Euclidean distances in $X$.

    Recall that for a matrix $A$, $\norm{A}_{\max}^2 = \left( \max{_{ij} \vert A_{ij} \vert } \right)^2 = \max{_{ij} \vert A_{ij} \vert ^2} \leq \sum_{ij} \vert A_{ij} \vert^2 = \norm{A}_F^2$, where $\norm{\cdot}_F$ denotes the Frobenius norm.
    Then we have 
    \[ \norm{D_X^{\circ^2} - D_Y^{\circ^2}}_{\max} \leq \norm{D_X^{\circ^2} - D_{Y}^{\circ^2}}_F.\]
    Define the eigendecomposition of the covariance of $X$ as $XX^T = Q \Lambda Q^T$, where $Q=[\vc_1, \dots, \vc_N]$ is the $N \times N$ matrix of orthonormal eigenvectors and $\Lambda = \diag( \lambda_1, \dots, \lambda_N )$ is the $N \times N$ diagonal matrix of eigenvalues corresponding to the orthogonal projection of $X$.
    Then the eigendecomposition of the $K$-dimensional subspace containing the top $K$ principal components of $X$ can be defined by $YY^T = Q \Lambda \vert_K Q^T$, where $\Lambda \vert_K = \diag( \lambda_1, \dots, \lambda_K,\underbrace{0,\dots,0}_{N-K} )$.
    Then, if we center the squared distances in $D_X^{\circ^2}$ and $D_Y^{\circ^2}$, we obtain the relations
    \begin{align*}
      XX^T  = -\frac{1}{2}C_N D_X^{\circ^2} C_N & ~\text{ and } \\
      YY^T  = -\frac{1}{2}C_N D_Y^{\circ^2} C_N,
    \end{align*}
    where $C_N$ is the $N \times N$ centering matrix with diagonal entries $1 - \frac{1}{N}$ and off-diagonal entries $-\frac{1}{N}$ \cite{anselin2020}.
    Then $D_X^{\circ^2} = -2 C_N (Q \Lambda Q^T) C_N$ and $D_Y^{\circ^2} = -2 C_N (Q \Lambda \vert_K Q^T) C_N$.
    Since $Q^TQ = QQ^T = I$, we obtain that
    \begin{align*}
      \norm{D_X^{\circ^2} - D_Y^{\circ^2}}_F & = 2  \norm{C_N (Q\Lambda Q^T - Q \Lambda \vert_KQ^T) C_N}_F \\
      & = 2 \norm{ Q(\Lambda - \Lambda \vert_K)Q^T}_F \quad \text{ (since }  C_N \approx I  \text{ for sufficiently large }  N\text{)} \\
      & = 2 \sqrt{\tr(Q (\Lambda - \Lambda\vert_K)^2 Q^T)} \\
      & = 2 \sqrt{\tr((\Lambda-\Lambda\vert_K)^2)} \quad \text{ (since trace is cyclically invariant)} \\
      & = 2 \sqrt{ \sum_{i=K+1}^N \lambda_i^2}. 
    \end{align*}
    Thus we get that $\dis(\phi) \leq \sqrt{2} \sqrt[4]{\sum_{i=K+1}^N \lambda_i^2}$ and hence $\dGH(X,Y) \leq \frac{\sqrt{2}}{2}\sqrt[4]{\sum_{i=K+1}^N\lambda_i^2}$.
    Then by Equation \ref{eq:VR_stability_PH}, we get that
    \[ \dB(\dgm_1(X),\dgm_1(Y)) \leq 2 \dGH(X,Y) \leq \sqrt{2} \sqrt[4]{\sum_{i=K+1}^N\lambda_i^2}. \]
    Hence by similar arguments as in the proof of Theorems \ref{thm:stability_conditional_score_change_f2} and \ref{thm:stability_conditional_score_change_f1}, we have that 
    \begin{align*}
      |\score(f_1|f_2) - \score_\phi(f_1|f_2)| &  = \frac{1}{\sqrt{3}} \left\vert \mp( \dgm_1(X) ) - \mp( \dgm_1(\phi(X))) \right\vert \\
        & \leq \frac{2}{\sqrt{3}} \dB( \dgm_1(X), \dgm_1(\phi(X))) \\
        & \leq 2\sqrt{\frac{2}{3}} \sqrt[4]{\sum_{i=K+1}^N\lambda_i^2}. \\
        & = \sqrt{\frac{8}{3}}\sqrt[4]{\sum_{i=K+1}^N\lambda_i^2}. \qedhere
    \end{align*}
\end{proof}

\begin{corollary}
    [Stability of Bi-Conditional Periodicity Score Under Dimension Reduction]\label{cor:stability_biconditional_score_und_PCA}
    Let $K \leq M + 1$ for $K \in \vN$.
    Suppose $f_1,f_2 : \vR \rightarrow \vR$ are continuous and differentiable univariate time seies.
    For $T \subseteq \vR$, define the eigenvalues obtained from the  orthogonal projection ($\phi$) of $\SW_{M,\tau_{w_2}}f_{1|2}(T)$ and $\SW_{M,\tau_{w_2}}f_{2|1}(T)$ onto their top $K$ principal components respectively by $\{ \lambda_i \}_{i=1}^N$ and $\{ \gamma_i \}_{i=1}^N$.
    Define the bi-conditional periodicity score of $f_1$ and $f_2$ under this orthogonal projection by $\score_\phi(f_1,f_2)$. 
    Then

    \[ |\score(f_1,f_2) - \score_\phi(f_1,f_2)| \leq \sqrt{\frac{2}{3}} \left( \sqrt[4]{\sum_{i=K+1}^N \lambda_i^2} + \sqrt[4]{\sum_{i=K+1}^N \gamma_i^2} \right). \]
\end{corollary}

\begin{proof}
    By Theorem \ref{thm:stability_conditional_score_und_PCA}, 

    \begin{align*}
        |\score(f_1,f_2) - \score_\phi(f_1,f_2)| & = \frac{1}{2} \left\vert \big( \score(f_1|f_2) + \score(f_2|f_1) \big) - 
        \big( \score_\phi(f_1|f_2) + \score_\phi(f_2|f_1) \big) \right\vert \\
        & \leq \frac{1}{2} \big( |\score(f_1|f_2) - \score_\phi(f_1|f_2)| + |\score(f_2|f_1) - \score_\phi(f_2|f_1)| \big) \\
        & \leq \frac{1}{2} \sqrt{\frac{8}{3}} \left( \sqrt[4]{\sum_{i=K+1}^N \lambda_i^2} + \sqrt[4]{\sum_{i=K+1}^N \gamma_i^2} \right) \\
        & = \sqrt{\frac{2}{3}} \left( \sqrt[4]{\sum_{i=K+1}^N \lambda_i^2} + \sqrt[4]{\sum_{i=K+1}^N \gamma_i^2} \right).
    \end{align*}
\end{proof}

\section{Experimental Results}\label{sec:exp_results}

\subsection{An Algorithm for the Conditional Periodicity Score}

We introduce a procedure for computing the conditional periodicity score of a discrete time series $f_i(T)$ given another $f_j(T)$, for $T = \{t_0,t_1,\dots,t_{N-1}\}_{N \in \mathbb{N}}$.
We then estimate the frequency of $f_j$ using the arg maximum of the spectral magnitude produced via the discrete fast Fourier transform. 
We particularly compute the discrete fast Fourier spectra of $f_j$, which we denote below by the absolute value of the dot product of the discrete signal with a pure sinusoidal with frequency $k$ cycles every $N$ time steps:

\begin{align*}
    | DFT[f_j](k) | = \frac{1}{N} \left\vert \sum_{n=0}^{N-1} (f_j)_n \cdot e^{-i 2\pi \cdot \frac{k}{N} \cdot n} \right\vert & = \frac{1}{N} \left\vert \sum_{n=0}^{N-1} (f_j)_n \cos \left( 2 \pi \frac{k}{N} n \right) - i \sum_{n=0}^{N-1} (f_j)_n \sin \left( 2 \pi \frac{k}{N} n \right) \right\vert \\
    & = \frac{1}{N} \left( \sum_{n=0}^{N-1} (f_j)_n \cos \left( 2 \pi \frac{k}{N} n \right) \right)^2 + \frac{1}{N} \left( \sum_{n=0}^{N-1} (f_j)_n \sin \left( 2 \pi \frac{k}{N} n \right) \right)^2. \\
\end{align*}

\noindent We then estimate $w_j$ as the frequency corresponding to the maximum spectral density of $f_2$, denoted by $w_j = \arg\max_k |DFT[f_j](k)|$.
We then use this estimate to define the time lag as $\tau_{w_j} = \frac{1}{w_j(M+1)}$ given an embedding dimension $M \in \vN$.
We then include an option for computing the top $K \in N$ principal components of the conditional sliding windows embedding $\SW_{M,\tau_{w_j}}f_{i|j}(T)$ under an orthogonal projection $\phi: \vR^{M+1} \rightarrow K$.
We finally perform Vietoris-Rips filtration on the conditional sliding windows embedding, obtaining the conditional periodicity score.

\begin{algorithm}[ht!]
\caption{Procedure for computing $\score(f_i|f_j)$}
\label{alg:pseudocode}
\textbf{Inputs:} An embedding dimension $M \in \vN$, two discrete univariate time series $\{f_i(T),f_j(T)\}$ for $T = \{t_0,t_1,\dots,t_{N-1}\}_{N \in \vN}$, PCA $\in \vR$.
\smallskip
\textbf{Estimate} $w_j$ via spectral analysis using the (Discrete) Fast Fourier Transform up to the Nyquist frequency:
\[ w_j = \underset{k=0,\dots,N/2-1}{\arg\max}|DFT[f_j](k)| \]

\textbf{Define} $\tau_{w_j} = \left\lfloor \frac{1}{w_j(M+1)} \right\rfloor$.

\textbf{Compute} $X = \SW_{M,\tau_{w_j}}f_{i|j}(T)$.

\textbf{If} PCA = 1,

\hspace*{0.3in} \textbf{For} $K = \min\{2,M+1\}$,

\hspace*{0.5in} \textbf{Compute} $\{ \phi(X) \}$ for orthogonal projection $\phi$ onto the first $K$ principle components of $X$.

\hspace*{0.5in} \textbf{Define} $Y = \phi(X)$.

\medskip

\textbf{Define} $Y = \phi(X)$.

\textbf{Center} and normalize $Y$ to obtain $Z$.

\smallskip

\textbf{Compute} $\dgm_1(Z)$ from Vietoris-Rips filtration on $Z$ and extract $\mp(\dgm_1(Z))$ from $\dgm_1(Z)$. 

\smallskip

\textbf{Compute} $\score(f_i|f_j)$ using $\mp(\dgm_1(Z))$. 
\medskip

\textbf{Return:} $\score(f_i|f_j)$.
\end{algorithm}

\begin{remark}[Computational Complexity of the Bi-Conditional Periodicity Score]\label{rmk:computational_complexity_score}
Algorithm \ref{alg:pseudocode} runs in  $O \Big( N \log N + PK^2 + P^6 \Big)$ time, where $N$ is the number of points in the discrete univariate input signals $f_i$ and $f_j$, $P = N - M \tau_{w_j}$ is the number of points in the conditional sliding windows embedding of $f_i$ given $f_j$, and $K < M+1$ is the number of principal components computed from the conditional sliding windows embedding for embedding dimension $M \in \vN$.
The discrete Fast Fourier Transform on $f_j$ can be computed in  $O(N \log N)$ time \cite{brigham1967}.
The PCA computations run in $O(PK^2)$ time \cite{Jo2002,yigit2024}.
The bottleneck step is usually the computation of the 1D persistence diagram using the VR filtration, which runs in $O(P^6)$ time \cite{KoMeRoTu2024}.

Notice that if we don't perform dimension reduction on the conditional embedding, then computing $\score(f_i|f_j)$ takes $O(N \log N + P^6)$ time.
Hence with PCA, the bi-conditional periodicity score, $\score(f_i,f_j)$, requires $O(N \log N + PK^2 + P^6)$ computations, and without PCA requires only $O(N \log N + P^6)$.
\end{remark}

\subsection{Synthetic Data}

\subsubsection{Stability of the Conditional Periodicity Score}

For our first experiment, we assess the stability of our conditional periodicity score when quantifying the extent to which varied frequencies $w_2$ are present in a time series $f_1$ with varied levels of added Gaussian noise.
We further compare this stability to the Fourier Spectrum of the noisy signal $f_1$ as well as percent recurrence.
We fix $f_1(t) = \sin(2\pi w_1 T)$ for $T = \{0,1,\dots,299\}$ and frequency $w_1 = 8/300$ Hz. 
We further define the noisy signal as $f_1^\sigma(T) = f_1(T) + \mathcal{N}(0,\sigma*\text{sd}(f_1))$, where $\text{sd}(f_1)$ denotes the standard deviation of $f_1$. 
We then vary the frequency $w_2$ of the reference sinusoidal $f_2$ among the values $w_2 = \{ k/300 \}$ for $k = 5, 5.25, \dots, 13$.
Additionally, we vary the standard deviation of Gaussian noise among the values $\sigma = \{0.01,0.05,0.1,0.15\}$ (i.e. 1\%, 5\%, 10\%, and 15\%).
For computing \%REC and $\score(f_1|f_2)$, we fix the embedding dimension to $M=1$. 
To compute \%REC, we additionally fix the time lag to $\tau = 10$ by taking the first minimum of the automutual information of $f_1$ \cite{wallot2018analyzing}. 
We finally fix the distance threshold to denote a cross-recurrence as $\tol = 0.1 \cdot \max(||\SW_{M,\tau}f_1(T) - \SW_{M,\tau} f_2(T)||)$ (i.e. 10\% of the maximum distance between the embeddings of $f_1$ and the reference sinusoidal).
We finally define our conditional periodicity score as 

\[ \score(f_1|f_2) = 1 - \frac{\mp \left( \dgm_1(\SW_{M,\tau_{w_2}}f_{1|2}(T) \right)}{\sqrt{3}}. \]

See our results in Figure \ref{fig:compare_condtional_score_stability_synthetic} below.
As expected, each measure shows a peak around the underlying frequency of $f_1$ for 0\% noise.
Additionally, our measure and the Fourier magnitude are much more robust to added Gaussian noise than \%REC is.
The stability of our conditional periodicity score and the spectral magnitude motivates us to use this magnitude to compute a stable measure of time series correlation, which we discuss below.

\begin{figure}[ht!]
    \centering
    \includegraphics[width=0.8\textwidth]{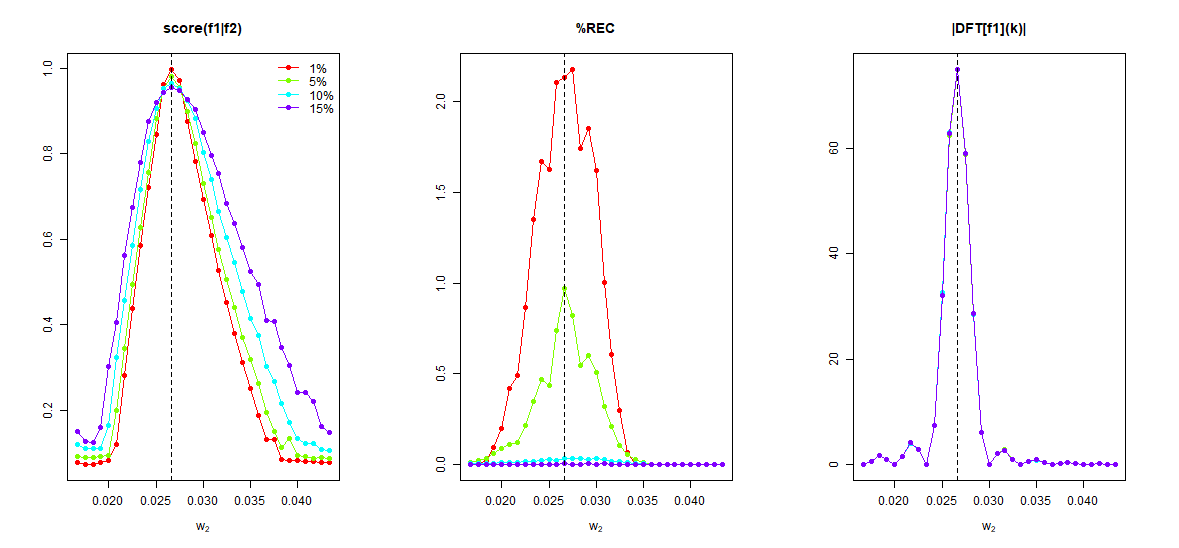}
    \caption{Given a time series $f_1$ of $N = 300$ points, we compare the stability of $\score(f_1|f_2)$, $|\text{REC}[f_1](k)|$ for $w_2 = k/N$, and \%REC$[f_1,e^{i 2\pi w_2 t}]$ when subject to added gaussian noise to $f_1$.
    We include a vertical dashed line at the underlying frequency $w_1$.}
    \label{fig:compare_condtional_score_stability_synthetic}
\end{figure}

\subsection{Stability of the Bi-Conditional Periodicity Score}
In this experiment, we compare the stability of $\score(f_1,f_2)$ as computed via Algorithm \ref{alg:pseudocode} to \%REC against added Gaussian noise to both $f_1$ and $f_2$. 
We again fix $f_1$ as before and vary the frequency of $f_2$ between the same values of $w_2$ as previously mentioned.
We again show the average of each measure over 20 samples of each signal $f_1^\sigma(t)$ and $f_2^\sigma(t)$ at each Gaussian noise level as described above. 
We define the embedding dimension as $M = 1$ for both measures.
For computing \%REC and \%DET, we fix the distance threshold to $\tol = 0.2 \cdot \max \{ || \SW_{M,\tau}f_1(T) - \SW_{M,\tau}f_2(T) ||\}$ and the time lag to be the first minimum of the automutual information of $f_1$ when $\sigma = 0$ (i.e. $\tau = 10$) \cite{wallot2018analyzing} .
To compute \%DET, we additionally fix the minimum number of diagonal points to $\minDL = 5$.

Notice in Figure \ref{fig:comparing_stability_biconditional_score_REC}, that our measure is much more stable to small levels of added Gaussian noise to both $f_1$ and $f_2$ than \%REC.
This suggests that our bi-conditional periodicity score is indeed a more stable measure of time-series correlation.

\begin{figure}[ht!]
\centering
\includegraphics[width=0.325\textwidth]{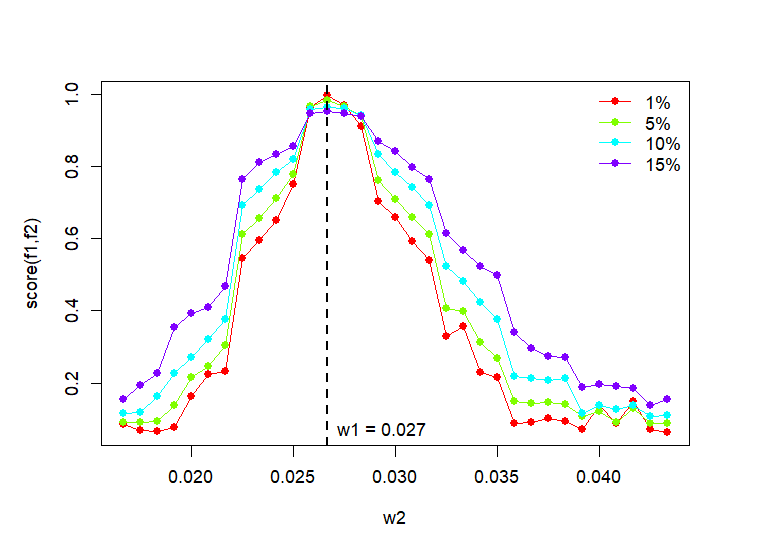}
\includegraphics[width=0.325\textwidth]{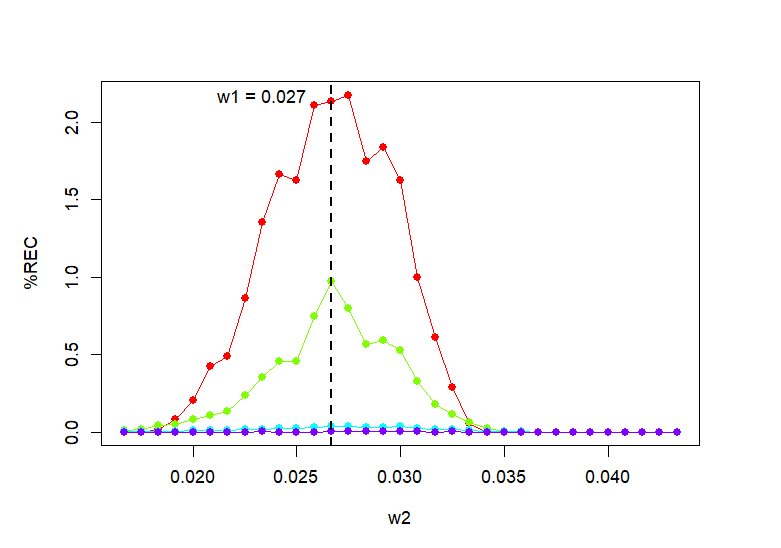}
\includegraphics[width=0.325\textwidth,height=0.155\textheight]{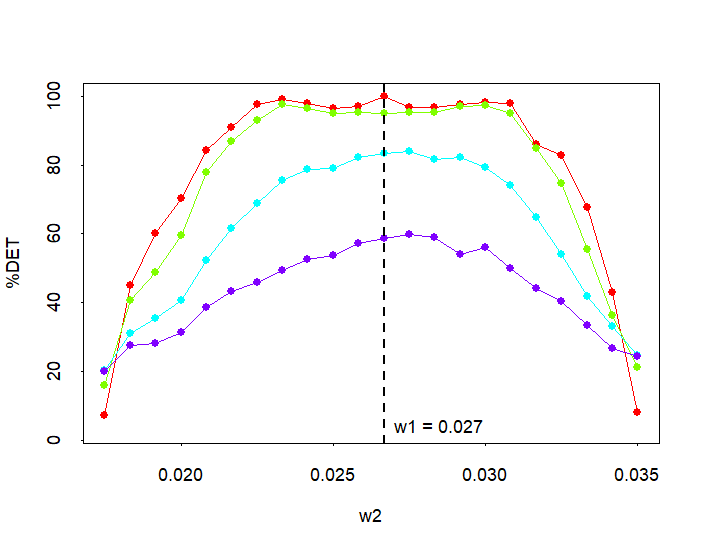}
\caption{Comparing the average stability of the bi-conditional periodicity score (left), percent recurrence (middle), and percent determinism (right) to added Gaussian noise to the sinusoidals $f_1(t) = \sin(2 \pi w_1 t)$ and $f_2 = \sin(2 \pi w_2 t)$.
We fix the frequency of $f_1$ to $w_1 = 8/300$ and vary that of $f_2$.
Then $\score(f_1,f_2)$ is the average of $\score(f_1|f_2)$ and $\score(f_2|f_1)$.
We denote the underlying frequency of $f_1$ via a vertical dashed line.}
\label{fig:comparing_stability_biconditional_score_REC}
\end{figure}

\subsection{Real Data}

For this experiment, we compare the similarity of climate data from three different locations near Australia.
We specifically compare three different time series of average hourly cooling temperatures (in degrees Farenheit) from the National Oceanic and Atmospheric Administration \cite{arguez_2010}.
The hourly cooling degree typically measures the energy demand (i.e. for air conditioning), by denoting how many degrees the average hourly temperature is above the baseline of $65^\circ$.
These hourly averages are taken over 30 years, between the years of 1981 and 2010.
We compare the corresponding time-series data from three weather stations - Pago Pago airport, Yap Island airport, and Koror weather service office.
In Figure \ref{fig:time_series_embeddings}, we show the respective locations of these stations on a map in black, blue, and green, along with their time series for the first 300 hours (12.5 weeks). 
We denote these time series as $f_\text{Pago}$ (black), $f_\text{Yap}$ (blue), and $f_\text{Koror}$ (green), which are defined on the hours $T = \{0,1,\dots,299\}$.
To pre-process these series, we first interpolate any missing values via cubic splining, and then fit a continuous function to each discrete signal evaluated at more discretized time points $T_\text{fine} = \{0,0.5,\dots,299\}$.

\begin{figure}[ht!]
\centering
\includegraphics[width=0.85\textwidth]{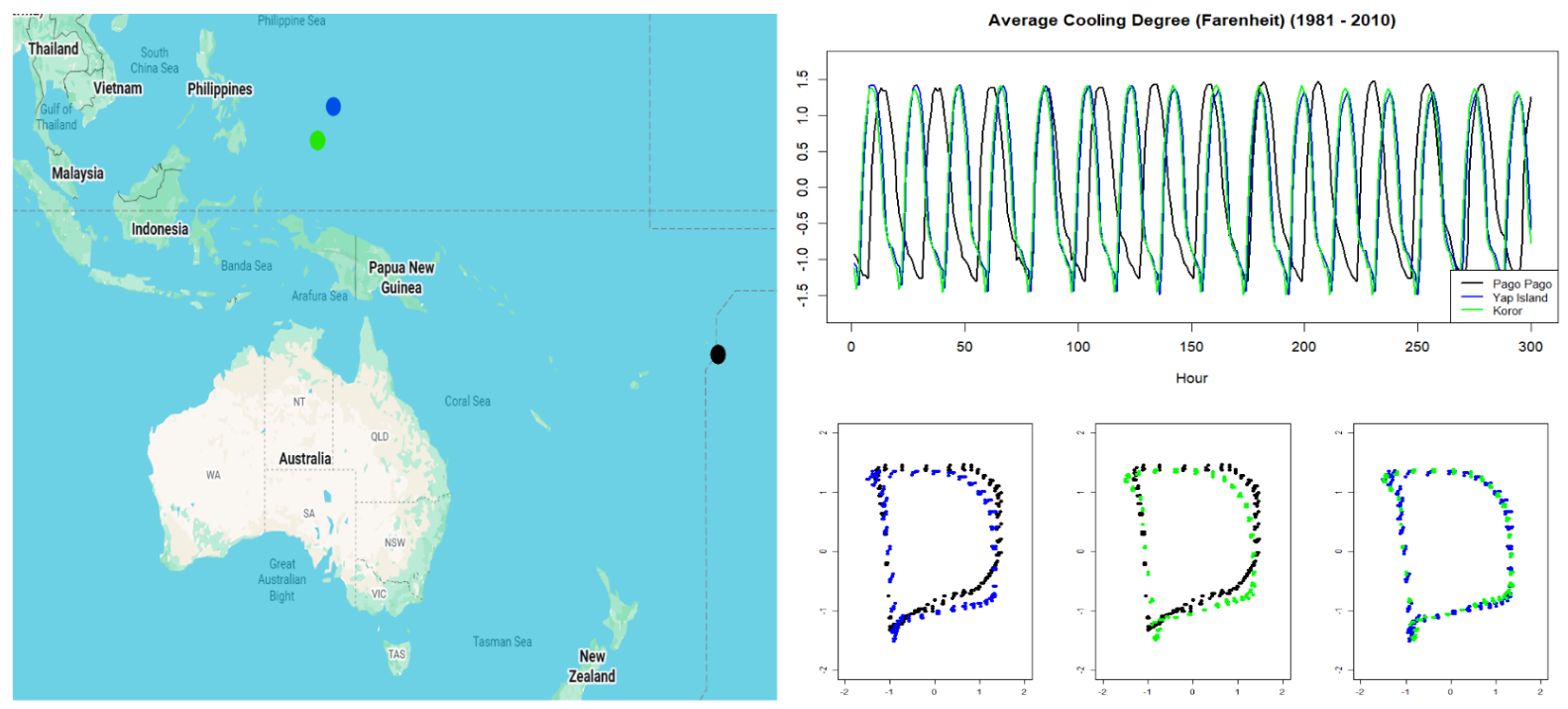}
\caption{Average Hourly Cooling Degrees recorded from three stations between 1981 and 2010, Pago Pago (black), Yap Island (blue), and Korer (green).
We show the locations of these stations on the left, the first 300 hours of data in the top right, and the corresponding pairs of time-series embeddings which we compare via \%DET in the bottom right.}
\label{fig:time_series_embeddings}
\end{figure}

In order to embed these time series, we fix the embedding dimension to $M = 1$.
The remaining parameters we define are used for computing \%DET, including the time lag $\tau_\text{DET} = 5$, the distance threshold of $\tol = c \cdot \max\{||\SW_{M,\tau}f_i(T_\text{fine}) - \SW_{M,\tau}f_j(T_\text{fine})||\}$ for $c = 0.00105$, and the minimum number of points to qualify a diagonal as $\minDL = 9$. 
We obtain the time lag by first extracting the first minimum of automutual information for each of the three time series ($\tau_\text{Pago} = 5$, $\tau_\text{Yap} = 4$, $\tau_\text{Koror}=5$) and then taking their maximum \cite{wallot2018analyzing}.
We denote their resulting embeddings as $\SW_{M,\tau}f_\text{Pago}(t)$, $\SW_{M,\tau}f_\text{Yap}(t)$, and $\SW_{M,\tau}f_\text{Koror}(t)$.
We obtain our distance threshold, $\tol$, as 0.105\% of the maximum pairwise distance between any two embeddings and the minimum $\minDL = 9$ by conducting a grid search of for each pair of series over the parameters $c = \{ 0.00104, 0.00105, 0.00106\}$ and $\minDL = \{8,9,10\}$ and obtaining stable measures of \%DET for all parametric combinations. 
When comparing $f_\text{Pago}$ and $f_\text{Yap}$, $f_\text{Pago}$ and $f_\text{Koror}$, and $f_\text{Yap}$ and $f_\text{Koror}$, we obtain \%DET values in the respective ranges $[5.11,15.9]$, $[9.2,18.61]$, and $[23.8,32.56]$ with respective averages of 10.33\%, 13.91\%, and 28.19\%.
We hence chose the median parameters of $c = 0.00105$ and $\minDL = 9$ for our computations. 
In Figure \ref{fig:time_series_embeddings}, we show all three pairs of time series embeddings for Pago Pago (black), Yap (blue), and Koror (green) which we compare using \%DET. 

In Figure \ref{fig:time_series_embeddings}, we can see that the average cooling degrees at Yap and Koror stations are more similar over time in comparison to Pago Pago. 
This makes sense intuitively, since these two stations are closer geographically, so we'd expect their oceanic environments to behave more similarly. 
As a result, we can also see that the embeddings of Yap and Koror data (shown in blue and green) also share more overlap than those of Yap and Pago Pago or Koror and Pago Pago.

\subsubsection{Comparing Frequency Estimates of Climate Data}

We first compare the frequency distributions produced from the discrete spectral magnitudes $|\text{DFT}[f_i](k)|$ for $k = \{10,11,\dots,40\}$ and $\score(f_i|f_w)$.
We specifically plot these measures over the frequencies $w = \{k/600\}_{k=10}^{40}$ in cycles per hour, and denote the corresponding reference sinusoidal for the conditional periodicity score to $f_w$.
We again fix the embedding dimension for $\score(f_i|f_w)$ to $M = 1$. 
We show the resulting distributions of $|\text{DFT}[f_i](k)|$ (top) and $\score(f_i|f_w)$ (bottom) in Figure \ref{fig:freq_distn_comparison}. 
We show the peaks (i.e. the frequency estimates) for Pago Pago, Yap, and Koror stations via purple vertical dashed lines.

\begin{figure}[ht!]
    \centering
    \includegraphics[width=0.85\linewidth]{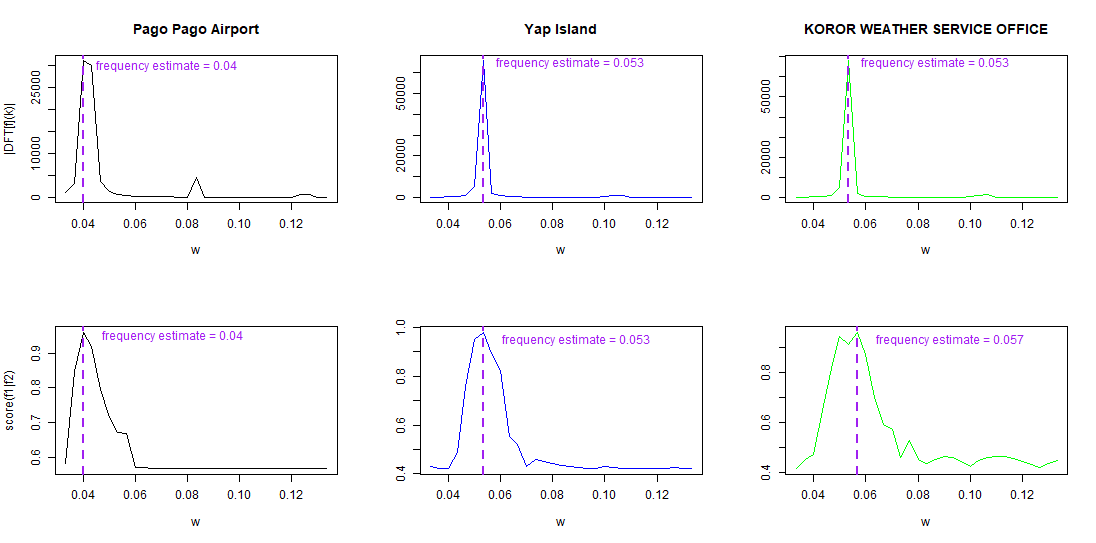}
    \caption{The frequency distributions of three time series of average hourly cooling degrees over thirty years as computed by the Discrete Fourier spectral magnitudes (top) and our conditional periodicity score (bottom).
    Yap Island and Koror both appear to contain temperatures that oscillate at a rate of about 19 hours per cycle, while Pago Pago temperatures oscillate at a slower rate of 25 hours a cycle.
    We fix the embedding dimension to compute $\score(f_1|f_2)$ to $M = 1$.}
    \label{fig:freq_distn_comparison}
\end{figure}

According to Algorithm \ref{alg:pseudocode}, we use the estimates as produced via $|\text{DFT}[f_i](k)|$, but we show the results from $\score(f_i|f_w)$ as well to highlight how close the estimates are for each. 
Using the estimates from $|\text{DFT}[f_i](k)|$, we fix the time lags for the three conditional sliding windows embeddings as $\tau_{\text{Pago}} = \frac{1}{0.04(M+1)}$, $\tau_{\text{Yap}} = \frac{1}{0.053(M+1)}$,
and $\tau_{\text{Koror}} = \frac{1}{0.053(M+1)}$. 
For example, the conditional sliding window embedding of Pago given $f_w$ is denoted $\SW_{M,\tau_w}f_{\text{Pago}|w} = \big( f_\text{Pago}(t),\dots,f_\text{Pago}(t+M\tau_w) \big)$, for $\tau_w = \frac{1}{w(M+1)}$ the time lag for the reference sinusoid $f_w$.

We also point out that the frequency estimates for both measures show greater similarity of Yap and Koror cooling degrees over time (about 19 hours per cycle) when compared to Pago (25 hours per cycle), further validating our intuition about their closeness geographically.

\subsubsection{Comparing Similarity Estimates of Climate Data}

Here we compare the correlation measures produced from percent determinism to our bi-conditional periodicity score for each combination of time series.
As previously mentioned, for both measures we fix the embedding dimension to $M = 1$, and the remaining parameters for \%DET we fix to $\tau_\text{DET} = 5$, $\minDL = 9$, and $\tol = 0.00105 \cdot \max\{ ||\SW_{M,\tau}f_i(t) - \SW_{M,\tau}f_j(t)||\}$.
We use the frequency estimates from the discrete Fourier magnitudes in Figure \ref{fig:freq_distn_comparison} to define the time lag of the conditional sliding window embeddings for each pair of time series comapred. 
For instance, the conditional sliding windows embedding and corresponding conditional periodicity score of Pago Pago given Yap are given below by 

\[ \SW_{M,\tau_\text{Yap}}f_{\text{Pago}|\text{Yap}}(t) = \big( f_\text{Pago}(t), \dots, f_\text{Pago}(t+M \tau_\text{Yap}) \big), \]

\[ \score(f_\text{Pago}|f_\text{Yap}) = 1 - \frac{\mp \left( \SW_{M,\tau_\text{Yap}}f_{\text{Pago}|\text{Yap}}(T_\text{fine}) \right)}{\sqrt{3}}, \]

\noindent where $\tau_\text{Yap} = \frac{1}{0.053(M+1)}$.
The bi-conditional periodicity score of Pago and Yap is then given as the average, $\score(f_\text{Pago},f_\text{Yap}) = \frac{\score(f_\text{Pago}|f_\text{Yap}) + \score(f_\text{Yap}|f_\text{Pago})}{2}$.
We show our results in Figure \ref{fig:results_correlation_climates}.

\begin{figure}[ht!]
    \centering
    \includegraphics[width=0.475\linewidth,height=0.225\textheight]{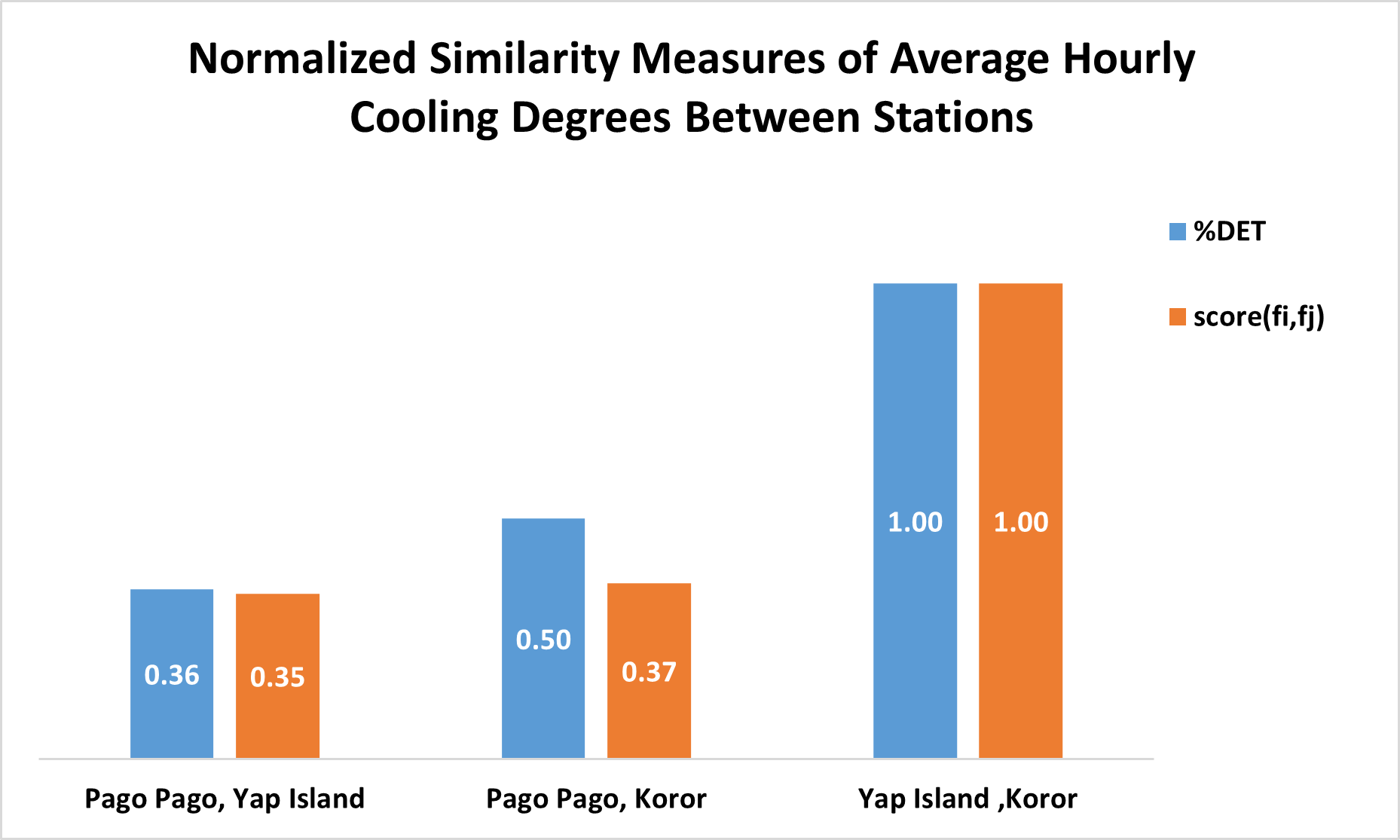}
    \includegraphics[width=0.45\textwidth]{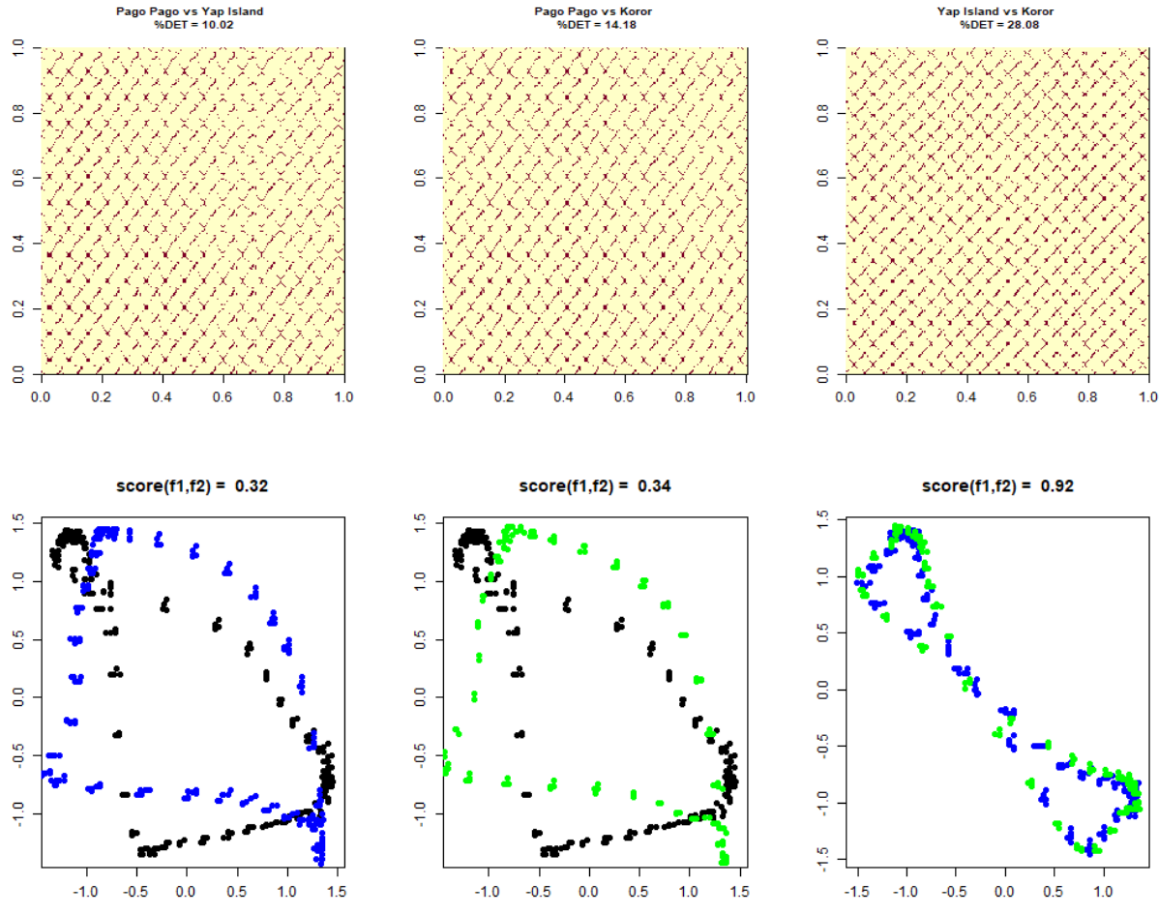}
    \caption{The similarity estimates of our biconditional periodicity score and percent determinism. 
    We show the normalized measures of \%DET in blue and $\score(f_1,f_2)$ on the left for each pair of time series.
    On the right we show the cross-recurrence matrices for each pair (top) and the corresponding conditional sliding windows embeddings (bottom).
    We fix the embedding dimension to $M = 1$ for both measures and color the conditional embeddings according to the series being embedded (i.e. the green embedding in the middle column corresponds to the conditional embedding of Koror given Pago).
    To compute \%DET, we fix the distance threshold to be 0.105\% of the maximum pairwise distance between both embeddings (i.e. $c = 0.00105$) and $\minDL = 9$.
    We obtain these values by retaining a stable range of values for \%DET among $c = \{0.00104,0.00105,0.00106\}$ and $\minDL = \{8,9,10\}$.
    We finally fix the time lag for \%DET by first identifying the first minimum of automutual information of each series and taking the maximum to be $\tau = 5$.}
    \label{fig:results_correlation_climates}
\end{figure}

We plot the normalized bi-conditional periodicity scores in orange and the normalized \%DET measures in blue. 
We can again see that both measures detect greater similarity in cooling degrees over time between the closer stations of Yap Island and Koror. 
In the right hand side of Figure \ref{fig:results_correlation_climates}, we show the cross-recurrence matrices and \%DET values for each pair of time series being compared (top), along with the corresponding conditional sliding windows embeddings used to compute the bi-conditional periodicity scores (bottom).
For example, the \%DET estimated when comparing the time series for Pago Pago and Yap Island is 10.04\%, obtained from the cross-recurrence matrix on the left.
The conditional sliding windows embeddings of Pago Pago given Yap and Yap given Pago Pago are shown directly below this matrix in black and blue, respectively, whose conditional periodicity scores are used to obtain the bi-conditional periodicity score of 0.32.
The normalized measures of \%DET and $\score(f_\text{Pago},f_\text{Yap})$ are then given as 0.36 and 0.35 in the left cluster of bars in the bar plot.

We again wish to highlight that both measures follow our intuition - since Pago Pago and Yap Island are stations that are closer together geographically, we naturally expect the behavior of their average cooling degrees over time to be more similar.
This greater similarity is highlighted in Figure \ref{fig:results_correlation_climates}, as both percent determinism and our bi-conditional periodicity score are largest when comparing these two time series. 

\subsubsection{Comparing Stability of Similarity Estimates on Climate Data}

We finally compare the stability of the bi-conditional periodicity score to percent determinism when adding increasing levels of Gaussian noise to each time series.
Here, we vary the standard deviation of Gaussian noise between the values $\sigma = \{0,0.01,\dots,0.09\}$ and report the average (normalized) measure over 30 samples of time series pairs at each noise level.
For example, we define the time series of Pago Pago at noise level $\sigma$ as 
$f_\text{Pago}^\sigma(t) = f_\text{Pago}(t) + \mathcal{N}(0,\sigma \cdot \text{sd}(f_\text{Pago}))$, where $\text{sd}(f_\text{Pago})$ denotes the standard deviation of the original time series.
We fix all parameters for both measures to those used in the previous experiment.
\label{fig:stability_results_climate_data}
See our results in Figure \ref{fig:stability_results_climate_data}.
As shown, for small amounts of added noise, our measure remains much more stable. 
Our bi-conditional periodicity score stays within 10\% of its original value in all three time-series comparisons for up to 10\% noise, while \%DET drops to less than 50\% of its original value for only 3\% noise.

\begin{figure}[ht!]
\centering
\includegraphics[width=0.65\textwidth]{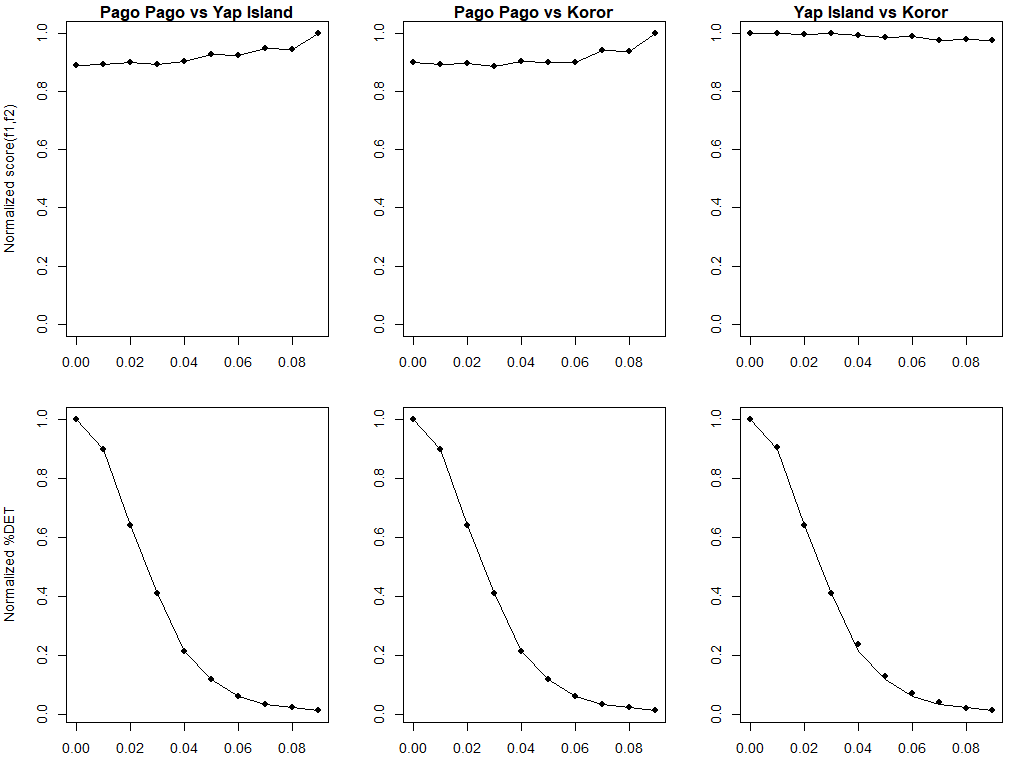}
\caption{The normalized average bi-conditional periodicity score (top) and percent determinism (bottom) for three pairs of time series containing average cooling degrees per hour against added levels of Gaussian noise.}
\label{fig:stability_results_climate_data}
\end{figure}

Notice that the conditional embeddings of Pago given Yap and Pago Given Koror have more clustered points (see the bottom right images of Figure \ref{fig:results_correlation_climates}).
Hence, adding small amounts of noise will spread them out more evenly, causing the conditional embeddings to be more rounded and the max persistence to increase (i.e. the score will increase).
This is why we see a small increase in the bi-conditional score in the top row of Figure \ref{fig:stability_results_climate_data}.
Oppositely, when comparing Yap and Koror, both pairs of conditional embeddings are much skinnier, hence adding small amounts of noise to their points will fill in the hole that is already there, causing a subtle decrease in the max persistence and hence the score.

\section{Discussion and Future Work}\label{sec:discussion_future}

In this work, we are successfully able to introduce a novel persistence-based measure of time-series similarity with both stability and parametric advantages to a previously defined method.
We are not only able to prove theoretical stability of our measure in comparison to percent determinism, but we are also able to experimentally verify this greater stability on synthetic and real time series data.
While this work is novel and our measure avoids the limitations posed by percent determinism, we are still interested in continuing this work by extending our measure to higher homology dimensions.
As well, we are interested in constructing a measure of similarity that uses a cross-recurrence matrix that measures correlation between the 0- (or higher-) dimensional Betti curves from filtration on a pair of time series.
As 0-dimensional Betti-curves are stable to small time-series perturbations, we are further motivated to try this approach as an alternatively stable measure of time-series similarity to percent determinism.

\bibliography{timeseries,homology,DimRedn,Refs_Prot,dynamics}

\end{document}